\documentclass[final]{amsart}
\usepackage{amssymb,amsmath, color,graphicx}

\usepackage{amssymb}

\usepackage[notref,notcite]{showkeys}

\newtheorem{Theorem}{Theorem}[section]
\newtheorem{lemma}[Theorem]{Lemma}

\newtheorem{proposition}[Theorem]{Proposition}
\theoremstyle{definition}
\newtheorem{remark}[Theorem]{Remark}

\newcommand{\R}{{\mathbb R}}

\newcommand{\cH}{{\mathcal H}}

\usepackage{epsfig}
\usepackage{geometry}
 \usepackage{pict2e}

\begin{document}

\title[Inertial dynamics and proximal algorithms]{Convergence  of inertial dynamics and proximal algorithms governed by maximally monotone operators}

\author{Hedy Attouch}
\address{Institut Montpelli\'erain Alexander Grothendieck, UMR 5149 CNRS, Universit\'e Montpellier, place Eug\`ene Bataillon,
34095 Montpellier cedex 5, France}
\email{hedy.attouch@univ-montp2.fr}

\author{Juan Peypouquet}
\address{Universidad T\'ecnica Federico Santa Mar\'\i a, Av Espa\~na 1680, Valpara\'\i so, Chile}
\email{juan.peypouquet@usm.cl}

\thanks{H. Attouch: Effort sponsored by the Air Force Office of Scientific Research, Air Force Material Command, USAF, under grant number F49550-1 5-1-0500. J. Peypouquet: supported by Fondecyt Grant 1140829, Millenium Nucleus ICM/FIC RC130003 and Basal Project CMM Universidad de Chile.}

\date{October 31, 2017}

\maketitle

\paragraph{\textbf{Abstract}} We study the behavior of the trajectories of a second-order differential equation with vanishing damping, governed by the Yosida regularization of a maximally monotone operator with time-varying index, along with a new {\em Regularized Inertial Proximal Algorithm} obtained by means of a convenient finite-difference discretization. These systems are the counterpart to accelerated forward-backward algorithms in the context of maximally monotone operators. A proper tuning of the parameters allows us to prove the weak convergence of the trajectories to zeroes of the operator. Moreover, it is possible to estimate the rate at which the speed and acceleration vanish. We also study the effect of perturbations or computational errors that leave the convergence properties unchanged. We also analyze a growth condition under which strong convergence can be guaranteed. A simple example shows the criticality of the assumptions on the Yosida approximation parameter, and allows us to illustrate the behavior of these systems compared with some of their close relatives.

\vspace{0.3cm}

\vspace{0.3cm}


\vspace{0.5cm}

\section{Introduction}

Let $\mathcal H$ be a real Hilbert space endowed with the scalar product $\langle \cdot,\cdot\rangle$ and norm $\|\cdot\|$. Given a maximally monotone operator $A: \mathcal H \rightarrow 2^{\mathcal H}$, we study the asymptotic behavior, as time goes to $+\infty$, of  the trajectories of the
second-order differential equation
\begin{equation} \label{E:system}
\ddot{x}(t)   + \frac{\alpha}{t}  \dot{x}(t)     +A_{\lambda(t)}(x(t)) = 0, \qquad t>t_0>0,
\end{equation}
where
$$
A_{\lambda} = \frac{1}{\lambda} \left( I- \left( I + \lambda A \right)^{-1} \right)
$$
stands for the Yosida regularization of $A$ of index $\lambda>0$ (see Appendix \ref{SS:Yosida} for its main properties), along with a new {\em Regularized Inertial Proximal Algorithm} obtained by means of a convenient finite-difference discretization of (\ref{E:system}). The design of rapidly convergent dynamics and algorithms to solve monotone inclusions of the form 
\begin{equation} \label{E:0inAx}
0\in Ax
\end{equation} 
is a difficult problem of fundamental importance in optimization, equilibrium theory, economics and game theory, partial differential equations, statistics, among other subjects. We shall come back to this point shortly. The dynamical systems studied here are closely related to Nesterov's acceleration scheme, whose rate of convergence for the function values is worst-case optimal. We shall see that, properly tuned, these systems converge to solutions of (\ref{E:0inAx}), and do so in a robust manner. We hope to open a broad avenue for further studies concerning related stochastic approximation and optimization methods. Let us begin by putting this in some more context. In all that follows, the set of solutions to (\ref{E:0inAx}) will be denoted by $S$.

\subsection{From the heavy ball to fast optimization}
Let $\Phi : \cH \to \mathbb R$ be a continuously differentiable convex function. The {\em heavy ball with friction} system
\begin{equation} \label{E:heavyball}
{\rm (HBF)}  \qquad    \ddot x(t)+\gamma\dot{x}(t)+\nabla\Phi (x(t))=0
\end{equation}
was first introduced, from an optimization perspective, by Polyak \cite{Polyak2}. The convergence of the trajectories of  (HBF) in the case of a convex function $\Phi$ has been obtained by \'Alvarez in \cite{Al}. In recent years, several studies  have been devoted to the study of the Inertial Gradient System ${\rm(IGS)_{\gamma}} $, with a time-dependent damping coefficient $\gamma(\cdot)$
\begin{equation} \label{E:IGS}
{\rm(IGS)_{\gamma}} \qquad \ddot x(t)+\gamma(t) \dot{x}(t)+\nabla\Phi (x(t))=0.
\end{equation}
A particularly interesting situation concerns the case $\gamma(t) \to 0$ of a vanishing damping coefficient. Indeed, as pointed out by Su, Boyd and Cand\`es in \cite{SBC}, the $\rm{(IGS)_{\gamma} }$ system with $\gamma (t) = \frac{\alpha}{t}$, namely:
\begin{equation}\label{E-FISTA}
\ddot{x}(t)   + \frac{\alpha}{t} \dot{x}(t) +    \nabla \Phi (x(t))  = 0,
\end{equation}
can be seen as a continuous version of the fast gradient
method of Nesterov (see \cite{Nest1,Nest2}), and its widely used successors, such as the Fast Iterative Shrinkage-Thresholding Algorithm (FISTA) of Beck-Teboulle \cite{BT}. When $\alpha\ge 3$, the rate of convergence of these methods is $\Phi(x_k)-\min_\cH \Phi= \mathcal O \left(\frac{1}{k^2}\right)$, where $k$ is the number of iterations.
Convergence of the trajectories generated by \eqref{E-FISTA}, and of the sequences generated by Nesterov's method, has been an elusive question for decades. However, when considering (\ref{E-FISTA}) with $\alpha >3$, it was shown by Attouch-Chbani-Peypouquet-Redont \cite{ACPR} and  May \cite{May},  that each trajectory converges weakly to an optimal solution, with the improved rate of convergence $\Phi(x(t))-\min_\cH \Phi = o (\frac{1}{t^2})$. Corresponding results for the algorithmic case have been obtained by Chambolle-Dossal \cite{CD} and Attouch-Peypouquet \cite{AP}.
Independently, and mainly motivated by applications to partial differential equations and control problems, Jendoubi-May \cite{JM} and Attouch-Cabot \cite{AC1,AC2} consider more general time-dependent damping coefficient $\gamma(\cdot)$. The latter includes the corresponding forward-backward algorithms, and unifies previous results.
In the case of a convex lower semicontinuous proper function
$\Phi: \mathcal H \rightarrow \R \cup + \{\infty \}$,
the ${\rm(IGS)_{\gamma}}$ dynamic (\ref{E:IGS}) is not well-posed, see \cite{ACRe}. A natural idea is to replace $\Phi$ by its Moreau envelope $\Phi_\lambda$, and consider 
\begin{equation}\label{E-FISTA3}
\ddot{x}(t)   + \frac{\alpha}{t} \dot{x}(t) +    \nabla \Phi_{\lambda (t)} (x(t))  = 0
\end{equation}
(see \cite{BC,Pey} for details, including the fact that $\nabla \Phi_\lambda=(\partial\Phi)_\lambda$). Fast minimization and convergence properties for the trajectories of (\ref{E-FISTA3}) and their related algorithms have been recently obtained by Attouch-Peypouquet-Redont \cite{AP-new}. 
This also furnishes a viable path towards its extension to the maximally monotone operator setting. It is both useful and natural to consider a time-dependent regularization parameter, as we shall explain now.

\subsection{Inertial dynamics and cocoercive operators} \label{SS:cocoercivity}
In analogy with (\ref{E:heavyball}), \'Alvarez-Attouch \cite{AA1} and Attouch-Maing\'e \cite{AM} studied the equation
\begin{equation} \label{E:AM1}
\ddot{x}(t) + \gamma  \dot{x}(t) + A(x(t)) = 0,
\end{equation}
where $A$ is a cocoercive operator. Cocoercivity plays an important role, not only to ensure the existence of solutions, but also in analyzing their long-term behavior. They discovered that it was possible to prove weak convergence to a solution of (\ref{E:0inAx}) if the cocoercivity parameter $\lambda$ and the damping coefficient $\gamma$ satisfy $\lambda \gamma^2 >1$. Taking into account that for $\lambda>0$, the operator $A_{\lambda}$ is $\lambda$-cocoercive and that $A_{\lambda}^{-1} (0) = A^{-1}(0)$ (see Appendix \ref{SS:Yosida}), we immediately deduce that, under the condition $\lambda \gamma^2 >1$, each trajectory of 
$$
\ddot{x}(t) + \gamma \dot{x}(t) + A_{\lambda}(x(t)) = 0
$$
converges weakly to a zero of $A$. In the quest for a faster convergence, our analysis of equation (\ref{E:system}), led us to introduce a time-dependent regularizing parameter $\lambda(\cdot)$ satisfying
$$\lambda(t)\times\frac{\alpha^2}{t^2}>1$$
for $t\ge t_0$. A similar condition appears in the study of the corresponding algorithms. We mention that a condition of the type $\lambda(t)\sim t^2$ appears to be critical to obtain fast convergence of the values in the case $A=\partial\Phi$, according to the results in \cite{AP-new}.
But system (\ref{E:system}) is not just an interesting extension of the heavy ball dynamic. It also arises naturally in stochastic variational analysis.

\subsection{Connections with stochastic optimization}
Let us present some links bewteen our approach and stochastic optimization.

\subsubsection{Stochastic approximation algorithms}
A close relationship between stochastic approximation algorithms and inertial dynamics with vanishing damping was established by Cabot, Engler and Gadat in \cite{CEG1}. Let us briefly (and skipping the technicalities) comment on this link and see how it naturally extends to our setting.

When $A$ is a sufficiently smooth operator, stochastic approximation algorithms are frequently used to approximate, with a random version of  the Euler's scheme, the behavior of the ordinary differential equation $\dot{x}(t) + A(x(t))=0$. If $A$ is a general maximally monotone operator, a natural idea is to apply this method to the regularized equation $\dot{x}(t) + A_{\lambda}(x(t))=0$, which has the same equilibrium points, since $A$ and  $A_{\lambda}$ have the same set of zeros. 
Consider first the case where $ \lambda>0$ is a fixed parameter. If we denote by $(X_n)_{n\in \mathbb N} $ the random approximants, $(w_n)_{n\geq 1}$ and $(\eta_n)_{n\geq 1}$ two auxiliary stochastic
processes, the recursive approximation is written as
\begin{equation}\label{stoch1}
X_{n+1} = X_n -\epsilon_{n+1}A_{\lambda}(X_n,w_{n+1}) + \epsilon_{n+1}\eta_{n+1},
\end{equation}
where $\epsilon_{n}$ is a sequence of positive real numbers, and $\eta_n$ is
a small residual perturbation. Under appropriate conditions 
(see \cite{CEG1}), solutions of  (\ref{stoch1})  asymptotically
behave like those of the deterministic differential equation 
$\dot{x}(t) + A_{\lambda}(x(t))=0$. A very common case occurs when $(w_n)_{n\geq 1}$  is a sequence of independent identically distributed variables with distribution $\mu$ and 
$$
A_{\lambda}(x) =\int A_{\lambda}(x, \omega) \mu (d\omega).
$$
When $A= \partial \Phi$ derives from a potential, this gives a stochastic optimization algorithm (see \cite{Rob_Mon}). When the random variable $A_{\lambda}(X_n,\cdot)$  has a large variance, the stochastic approximation of $A_{\lambda}(X_n)$ by $A_{\lambda}(X_n,w_{n+1})$ can be numerically improved by using the following modified recursive definition:
\begin{equation}\label{stoch2}
X_{n+1} = X_n -\frac{\epsilon_{n+1}}{\sum \epsilon_i}\sum_i \epsilon_i A_{\lambda}(X_i,w_{i+1}) .
\end{equation}
As proved in \cite[Appendix A]{CEG1}, the limit ordinary differential equation has the form 
\begin{equation}\label{stoch3}
\ddot{X}(t) + \frac{1}{t + b} \dot{X}(t) + \frac{1}{t + b}A_{\lambda} (X(t))=0,
\end{equation}
where $b$ is a positive parameter. After the successive changes of the time variable $t \mapsto t-b$, 
$t \mapsto 2\sqrt{t}$, we finally obtain
\begin{equation}\label{stoch4}
\ddot{X}(t) + \frac{1}{t} \dot{X}(t) + A_{\lambda} (X(t))=0,
\end{equation}
which is precisely \eqref{E:system} with $\alpha=1$ and $\lambda(t)\equiv\lambda$.

This opens a number of possible lines of research: First, it would be interesting to see how the coefficient $ \alpha $ will appear in the stochastic case. Next, it appears important to understand the connection between \eqref{stoch3} and \eqref{stoch4} with a time-dependent coefficient $ \lambda (\cdot) $. Combining these two developments, we could be able to extend the stochastic approximation method to a wide range of equilibrium problems, and expect that the fast convergence results to hold, considering that \eqref{E:system} is an accelerated method in the case of convex minimization when $ \alpha \geq 3$. In the case of a gradient operator, the above results also suggest a natural link with the epigraphic law of large numbers of Attouch and Wets \cite {AW-LLN}. The analysis can be enriched using the equivalence between epi-convergence of a sequence of functions and the convergence of the associated evolutionary semigroups.

\subsubsection{Robust optimization, regularization}

Suppose we are interested in finding a zero of an operator $A$, which is not exactly known (a situation that is commonly encountered in inverse problems, for instance). If the uncertainty follows a {\it known} distribution, then the stochastic approximation approach described above may be applicable. Otherwise, if $A$ is known to be sufficiently close to some model $\hat A$ (in a sense to be precised below), an alternative is to use robust analysis techniques, interpreting $A$ as a perturbation of $\hat A$. Regularization techniques (like Tikhonov's, where the operator $A$ is replaced by a regularized operator $A+ \rho B$) follow a similar logic. We recall the notion of graph-distance bewteen two operators, introduced by Attouch and Wets in \cite{AW1, AW2} (see also \cite{RW}). It can be formulated using Yosida approximations as
$$
d_{\lambda, \rho} (A, \hat A) = \sup_{\|x\|\leq \rho} \|  A_{\lambda}(x) - \hat A_{\lambda}(x)\|,
$$
where $\rho$ is a positive parameter. These pseudo-distances are particularly well adapted to our dynamics, which is expressed using Yosida approximations. In the case of minimization problems, they are associated with the notion of epi-distance. The calculus rules developed in \cite{AW1, AW2} make it possible to estimate these distances in the case of operators with an additive or composite structure. In Section \ref{Stab}, the convergence of the algorithm is examined in the case of errors that go asymptotically to zero. A more general situation, where approximation and iterative methods are coupled in the presence of noise, is an ongoing research topic. See  \cite{SLB,MRSV,VSBV} for recent results in this direction.

\subsection{Organization of the paper}
In Section \ref{S:minimizing}, we study the asymptotic convergence properties of the trajectories of the continuous dynamics \eqref{E:system}.
We prove that each trajectory converges weakly, as $t\to+\infty$, to a solution of \eqref{E:0inAx}, with $\|\dot x(t)\|=O(1/t)$ and $\|\ddot x(t)\|=O(1/t^2)$. Then, in Section \ref{S:FIPA}, we consider the corresponding proximal-based inertial algorithms and prove parallel convergence results. The effect of external perturbations on both the continuous-time system and the corresponding algorithms is analyzed in 
Section \ref{Stab}, yielding the robustness of both methods. In Section \ref{strong}, we describe how convergence is improved under a quadratic growth condition. Finally, in Section \ref{rot}, we give an example that shows the sharpness of the results and illustrates the behavior of (\ref{E:system}) compared to other related systems. Some auxiliary technical results are gathered in an Appendix.

\section{Convergence results for the continuous-time system}\label{S:minimizing}

In this section, we shall study the asymptotic properties of the continuous-time system:
\begin{equation} \label{E:system_bis}
\ddot{x}(t)   + \frac{\alpha}{t}  \dot{x}(t)     +A_{\lambda(t)}(x(t)) = 0,\qquad t>t_0>0.
\end{equation}
Although some of the results presented in this paper are valid when $ \lambda (\cdot) $ is locally integrable, we shall assume $ \lambda (\cdot)$ to be continuous, for simplicity. The function $(t,x)\mapsto A_{\lambda(t)}(x)$ is continuous in $(t,x)$, and uniformly Lipschitz-continuous with respect to $x$. This makes \eqref{E:system_bis} a classical differential equation, whose solutions are unique, and defined for all $t\ge t_0$, for every initial condition $x(t_0)=x_0$, $\dot x(t_0)=v_0$ \footnote {The idea consisting in regularizing with the help of the Moreau envelopes an inertial dynamic governed by a nonsmooth operator was already used in the modeling of elastic shocks in \cite{ACRe}}. See Lemma (\ref{existence-uniqueness}) for the corresponding result in the case where $\lambda (\cdot)$ is just locally integrable.

The main result of this section is the following:

\begin{Theorem}\label{T:weak_convergence}
Let $A: \mathcal H \rightarrow 2^{\mathcal H}$ be  a maximally monotone operator such that  $S= A^{-1} (0)\neq \emptyset$. Let $x:[t_0,+\infty[\to\mathcal H$ be a solution of (\ref{E:system_bis}), where 
$$\alpha>2\quad \hbox{and}\quad\lambda (t) = (1+\epsilon) \frac{t^2}{\alpha^2} \quad \hbox{for some}\quad \epsilon > \frac{2}{\alpha -2}.$$ 
Then, $x(t)$ converges weakly, as $t\to+\infty$, to an element of $S$. Moreover $\lim_{t \to + \infty} \| \dot{x}(t)\|  =\lim_{t \to + \infty} \| \ddot{x}(t) \| =0.$
\end{Theorem}

\subsection{Anchor}

As a fundamental tool we will use the distance to equilibria in order to anchor the trajectory to the solution set $S=A^{-1}(0)$. To this end, given a solution trajectory $x:[t_0,+\infty[\to\mathcal H$ of (\ref{E:system_bis}), and a point $z\in \cH$, we define $h_z:[t_0,+\infty [ \to\R$ by
\begin{equation} \label{E:h_z}
h_z(t)=\frac{1}{2}\|x(t)-z\|^2.
\end{equation} 

We have the following:

\begin{lemma} \label{L:h_z}
Given $z\in S= A^{-1} (0)$, define $h_z$ by (\ref{E:h_z}). For all $t \geq t_0$, we have
\begin{equation} \label{E:ineq_h_z}
\ddot{h}_z(t) + \frac{\alpha}{t} \dot{h}_z(t) + \lambda (t) \| A_{\lambda(t)}(x(t))  \|^2 \leq \|\dot{x}(t)\|^2.
\end{equation} 
\end{lemma}

\begin{proof}
First observe that
$$\dot{h}_z(t) = \langle x(t) - z , \dot{x}(t)  \rangle\qquad\hbox{and}\qquad 
\ddot{h}_z(t) = \langle x(t) - z , \ddot{x}(t)  \rangle + \| \dot{x}(t) \|^2.$$
By \eqref{E:system_bis}, it ensues that
\begin{equation} \label{wconv3}
\ddot{h}_z(t) + \frac{\alpha}{t} \dot{h}_z(t) +
\langle A_{\lambda(t)}(x(t)), x(t) - z  \rangle = \| \dot{x}(t) \|^2 .
\end{equation}
Since $z\in S = A^{-1} (0)$, we have  $A_{\lambda(t)}(z)=0$, and the $\lambda(t)$-cocoercivity of $A_{\lambda(t)}$ gives
\begin{equation} \label{wconv4}
\langle A_{\lambda(t)}(x(t)), x(t) - z  \rangle \geq \lambda (t)  \| A_{\lambda(t)}(x(t)) \|^2 .
\end{equation}
It suffices to combine (\ref{wconv3}) and (\ref{wconv4}) to conclude.
\end{proof}
As suggested in \ref{SS:cocoercivity}, we are likely to need a growth assumption on $\lambda (t)$ $-$related to $\lambda(t) \times \left(\frac{\alpha}{t}\right)^2>1$$-$ in order to go further. Whence, the following result:

\begin{lemma} \label{L:h_z-2}
Given $z\in S= A^{-1} (0)$, define $h_z$ by (\ref{E:h_z}) and let 
\begin{equation} \label{E:basic-assumption}
\lambda (t) \frac{\alpha^2}{t^2} \geq 1+ \epsilon
\end{equation} 
for some $\epsilon >0$. Then, for each $z\in S= A^{-1} (0)$ and all $t \geq t_0$, we have
\begin{equation} \label{E:ineq_h_z-2}
\ddot{h}_z(t) + \frac{\alpha}{t} \dot{h}_z(t) + \epsilon \|\dot{x}(t)\|^2   +  \frac{\alpha\,\lambda (t)}{t} \frac{d}{dt} \|\dot{x}(t)\|^2 + \lambda(t) \|\ddot{x}(t)\|^2 \leq 0.
\end{equation} 
\end{lemma}
\begin{proof} 
By (\ref{E:system_bis}), we have 
$$
A_{\lambda(t)}(x(t)) = -\ddot{x}(t)  - \frac{\alpha}{t}  \dot{x}(t). 
$$
Replacing $A_{\lambda(t)}(x(t))$ by this expression in
(\ref{E:ineq_h_z}), we obtain 
\begin{equation} \label{E:ineq_h_z-100}
\ddot{h}_z(t) + \frac{\alpha}{t} \dot{h}_z(t) + \lambda (t) \left\| \ddot{x}(t)  + \frac{\alpha}{t}  \dot{x}(t)   \right\|^2 \leq \|\dot{x}(t)\|^2.
\end{equation} 
After expanding the third term on the left-hand side of this expression, we obtain 
$$\ddot{h}_z(t) + \frac{\alpha}{t} \dot{h}_z(t) + \left(\lambda (t) \frac{\alpha^2}{t^2} -1   \right) \|\dot{x}(t)\|^2   + \alpha \frac{\lambda (t)}{t} \frac{d}{dt} \|\dot{x}(t)\|^2 + \lambda(t) \|\ddot{x}(t)\|^2 \leq 0.
$$
The result follows from (\ref{E:basic-assumption}).
\end{proof}

\subsection{Speed and acceleration decay}

We now focus on the long-term behavior of the speed and acceleration. We have the following:

\begin{proposition} \label{P:system_properties}
Let $A: \mathcal H \rightarrow 2^{\mathcal H}$ be  a maximally monotone operator such that  $S= A^{-1} (0)\neq \emptyset$. Let $x:[t_0,+\infty[\to\mathcal H$ be a solution of (\ref{E:system_bis}), where 
$$\alpha>2\quad \hbox{and}\quad\lambda (t) = (1+\epsilon) \frac{t^2}{\alpha^2} \quad \hbox{for some}\quad \epsilon > \frac{2}{\alpha -2}.$$ 
Then, the trajectory $x(\cdot)$ is bounded, $\|\dot x(t)\|=\mathcal O(1/t)$, $\|\ddot x(t)\|=\mathcal O(1/t^2)$ and
$$\int_{t_0}^{+\infty} t\|\dot x(t)\|^2\,dt<+\infty.$$
\end{proposition}

\begin{proof}
As before, take $z\in S= A^{-1} (0)$ and define $h_z$ by (\ref{E:h_z}). First we simplify the writing of equation (\ref{E:ineq_h_z-2}), given in Lemma \ref{L:h_z-2}. By setting $g(t) = \|\dot{x}(t)\|^2$, and $\beta = \frac{1 + \epsilon}{\alpha}$, we  have 
\begin{equation} \label{E:ineq_h_z-4}
\ddot{h}_z(t) + \frac{\alpha}{t} \dot{h}_z(t) + \epsilon g(t)   + \beta t \dot{g}(t)+ \lambda(t) \|\ddot{x}(t)\|^2 \leq 0.
\end{equation} 
Neglecting the positive term $\lambda(t) \|\ddot{x}(t)\|^2$ and multiplying by $t$, we obtain
\begin{equation} \label{E:ineq_h_z-6}
t\ddot{h}_z(t) + \alpha \dot{h}_z(t) + \epsilon t g(t)   + \beta t^2 \dot{g}(t) \leq 0.
\end{equation}
Integrate this inequality from $t_0$ to $t$. We have
\begin{align*}
\int_{t_0}^t \left( s\ddot{h}_z(s) + \alpha \dot{h}_z(s)\right) ds &= t\dot{h}_z(t)-t_0\dot{h}_z(t_0) -\int_{t_0}^t \dot{h}_z(s)ds + \alpha h_z(t) -  \alpha h_z(t_0) \\
&= t\dot{h}_z(t) +(\alpha -1)h_z(t) -t_0\dot{h}_z(t_0)-(\alpha -1)h_z(t_0) 
\end{align*}
\begin{align*}
\int_{t_0}^t \left( \epsilon s g(s)   + \beta s^2 \dot{g}(s) \right) ds &= \beta t^2 g(t) - \beta {t_0}^2 g(t_0) - 2 \beta \int_{t_0}^t s g(s) ds + \epsilon \int_{t_0}^t s g(s) ds\\
&= \beta t^2 g(t) - \beta {t_0}^2 g(t_0)  + (\epsilon - 2 \beta)\int_{t_0}^t s g(s) ds .
\end{align*}
Adding the two above expressions, we obtain
\begin{equation} \label{E:ineq_h_z-7}
t\dot{h}_z(t) +(\alpha -1)h_z(t) + \beta t^2 g(t) + (\epsilon - 2 \beta)\int_{t_0}^t s g(s) ds \leq C
\end{equation}
for some positive constant $C$ that depends only on the initial data.
By the definition of $\beta= \frac{1 + \epsilon}{\alpha}$, we have $\epsilon -2 \beta = \frac{\alpha -2}{\alpha}\left(\epsilon -  \frac{2}{\alpha -2} \right)$.
Since $\alpha >2$ and $\epsilon > \frac{2}{\alpha -2}$, we observe that $\epsilon - 2 \beta >0$. Hence, (\ref{E:ineq_h_z-7}) gives
\begin{equation} \label{E:ineq_h_z-8}
t\dot{h}_z(t) +(\alpha -1)h_z(t)  \leq C.
\end{equation}
Multiply this expression by $t^{\alpha -2}$ to obtain
$$
\frac{d}{dt} t^{\alpha -1}h(t) \leq Ct^{\alpha -2}.
$$
Integrating from $t_0$ to $t$, we obtain
$$
h_z(t) \leq \frac{C}{\alpha -1} + \frac{D}{t^{\alpha -1}},
$$
for some other positive constant $D$. 
As a consequence, $h_z(\cdot)$ is bounded, and so is the trajectory $x(\cdot)$. Set 
$$
M := \sup_{t\geq t_0} \|x(t)\| < + \infty,
$$
and note that
\begin{equation} \label{E:ineq_h_z-9}
|\dot{h}_z(t)| =  |\langle x(t) - z , \dot{x}(t)  \rangle|\leq \|x(t) - z\| \| \dot{x}(t) \| \leq ( M + \|z\|)\|\dot{x}(t)\|. 
\end{equation}
Combining (\ref{E:ineq_h_z-7}) with (\ref{E:ineq_h_z-9}) we deduce that
$$
\beta \big(t\|\dot{x}(t)\|\big)^2 \leq C + ( M + \|z\|)\,\big(t\|\dot{x}(t)\|\big). 
$$
This immediately implies that
\begin{equation} \label{E:ineq_h_z-10}
\sup_{t\geq t_0} t\|\dot{x}(t)\| < + \infty.
\end{equation}
This gives 
\begin{equation} \label{E:dotx_1/t}
\|\dot x(t)\|=\mathcal O(1/t).
\end{equation}	
From \eqref{E:system_bis}, we have 
\begin{equation} \label{E:system-cor}
\|\ddot{x}(t) \|  \leq  \frac{\alpha}{t}  \|\dot{x}(t)    \| + \| A_{\lambda(t)}(x(t))\| \leq \frac{\alpha}{t} \|\dot{x}(t)    \| + \frac{M+\|z\|}{\lambda (t)}. 
\end{equation}
Using (\ref{E:dotx_1/t}) and the definition of $\lambda(t)$, we conclude that
\begin{equation} \label{E:ddotx_1/t2}
\|\ddot x(t)\|=\mathcal O(1/t^2).
\end{equation}		
Finally, returning to (\ref{E:ineq_h_z-7}), and using (\ref{E:ineq_h_z-9}) and (\ref{E:ineq_h_z-10}) we infer that
\begin{equation} \label{E:ineq_h_z-11}
\int_{t_0}^{+ \infty} t \|\dot{x}(t)\|^2 dt < + \infty,
\end{equation}
which completes the proof.
\end{proof}

\subsection{Proof of Theorem \ref{T:weak_convergence}}
We are now in a position to prove the convergence of the trajectories to equilibria. According to Opial's Lemma \ref{Opial}, it suffices to verify that $\lim_{t\to + \infty}\|x(t)-z\|$ exists for each $z\in S$, and that every weak limit point of $x(t)$, as $t\to+\infty$, belongs to $S$.

\smallskip

We begin by proving that $\lim_{t\to\infty}\|x(t)-z\|$ exists for every $z\in S$. By Lemma \ref{L:h_z}, we have
\begin{equation} \label{E:ineq_h_z-13}
t\ddot{h}_z(t) + \alpha \dot{h}_z(t) + t\lambda (t) \| A_{\lambda(t)}(x(t))  \|^2 \leq t\|\dot{x}(t)\|^2.
\end{equation} 
Since $ h_z$ is nonnegative, and in view of (\ref{E:ineq_h_z-11}), 
Lemma \ref{basic-edo} shows that
$$
\lim_{t\to + \infty} h_z (t)
$$ 
exists. It follows that $\lim_{t\to\infty}\|x(t)-z\|$ exists for every $z\in S$.

\smallskip

From Lemma \ref{basic-edo} and (\ref{E:ineq_h_z-13}), we also have
\begin{equation} \label{E:ineq_h_z-14}
\int_{t_0}^{+ \infty} t \lambda (t) \| A_{\lambda(t)}(x(t))  \|^2  dt < + \infty .
\end{equation}
Since $\lambda (t) = ct^2$ for some positive constant $c$ we infer
\begin{equation} \label{E:ineq_h_z-15}
\int_{t_0}^{+ \infty}  \|\lambda(t) A_{\lambda(t)}(x(t))  \|^2  \frac{1}{t}dt < + \infty .
\end{equation}
The central point of the proof is to show that this property implies
\begin{equation} \label{E:ineq_h_z-16}
\lim_{t\to + \infty}  \|\lambda(t) A_{\lambda(t)}(x(t))  \|=0.
\end{equation}
Suppose, for a moment, that this property holds. Then the end of the proof follows easily:  let  $t_n \to + \infty$ such that 
$x(t_n) \to \bar{x}$ weakly. We have $\lambda(t_n) A_{\lambda(t_n)}(x(t_n)) \to 0$ strongly. Since $\lambda(t_n)\to + \infty$, we also have  $ A_{\lambda(t_n)}(x(t_n)) \to 0$ strongly.
Passing to the limit in
$$
A_{\lambda(t_n)}(x(t_n)) \in A (x(t_n) - \lambda(t_n) A_{\lambda(t_n)}(x(t_n)))
$$
and using the demi-closedness of the graph of $A$, we obtain 
$$
0 \in A (\bar{x}).
$$
In other words, $\bar{x} \in S$, and we conclude by Opial's Lemma.

\smallskip

As a consequence, it suffices to prove (\ref{E:ineq_h_z-16}). To obtain this result, we shall estimate the variation of the function $t \mapsto \lambda(t) A_{\lambda(t)} $. By applying Lemma \ref{L:basic-variation} with $\gamma= \lambda(t)$, $\delta= \lambda (s)$, $x= x(t)$ and $y = y(s)$ with $s, t \geq t_0$, we obtain 
\begin{equation} \label{E:basic-var2}
\|\lambda(t) A_{\lambda(t)}x(t) - \lambda(s) A_{\lambda (s)}x(s)\| \leq 2 \| x(t)-x(s) \| + 2 \| x(t)-z \|\frac{|\lambda(t) - \lambda (s) |}{\lambda(t)}
\end{equation} 
for each fixed $z\in S$. Dividing by $t-s$ with $t\neq s$, and letting $s$ tend to $t$, we deduce that
$$
\left\|\frac{d}{dt}\big(\lambda(t) A_{\lambda(t)}x(t)\big)\right\| \leq 2 \|\dot{x}(t)\| +  2 \| x(t)-z \|\frac{|\dot{\lambda}(t) |}{\lambda(t)}
$$
for almost every $t>t_0$. According to Proposition \ref{P:system_properties}, the trajectory $x(\cdot)$ is bounded by some $M\ge 0$. In turn, 
$$\|\dot{x}(t)\| \leq \frac{C}{t},$$
for some $C\ge 0$ and all $t\ge t_0$, by \eqref{E:ineq_h_z-10}. Finally, the definition of $\lambda(t)$ implies 
$$\frac{|\dot \lambda(t)|}{\lambda(t)}=\frac{2}{t}$$
for all $t\ge t_0$. As a consequence, 
$$
\left\|\frac{d}{dt}(\lambda(t) A_{\lambda(t)}x(t))\right\| \leq \frac{2C+4(M+\|z\|)}{t}.
$$ 
%
This property, along with the boundedness of $\lambda(t) A_{\lambda(t)}x(t)$, and estimation \eqref{E:ineq_h_z-15}, together imply that the nonnegative function  $w(t):= \| \lambda(t) A_{\lambda(t)}x(t)\|^2$ satisfies 

$$\left|\frac{d}{dt} w(t)\right| \leq \eta(t)$$
for almost every $t>t_0$, and
$$\int_{t_0}^{+ \infty}  w(t)\,\eta(t)\,dt < + \infty,$$
where
$$\eta(t)=\frac{2C+4(M+\|z\|)}{t}.$$
Noting that $\eta\notin L^1 (t_0, +\infty)$, we  conclude thanks to Lemma \ref{L:basic-integration}.

\section{Proximal-based inertial algorithms with regularized operator}\label{S:FIPA}

In this section, we introduce the inertial proximal algorithm which results from the discretization with respect to time of the continuous system
\begin{equation} \label{E:prox1}
\ddot x(t)+\frac{\alpha}{t} \dot{x}(t)+A_{\lambda(t)}(x(t))=0,
\end{equation}
where $A_{\lambda(t)}$ is the Yosida approximation of index $\lambda(t)$  of the maximally monotone operator $A$. Further insight into the relationship between continuous- and discrete-time systems in variational analysis can be found in \cite{Pey_Sor}.

\subsection{A Regularized Inertial Proximal Algorithm}

We shall obtain an implementable algorithm by means of an {\em implicit} discretization of \eqref{E:prox1} with respect to time. Note that, in view of the Lipschitz continuity property of the operator $A_{\lambda}$, the explicit discretization might work well too.  We choose to discretize it implicitely for two reasons: implicit discretizations tend to follow the continuous-time trajectories more closely; and the explicit discretization has the same iteration complexity (they each need one resolvent computation per iteration). Taking a fixed time step $h>0$, and setting $t_k= kh$, $x_k = x(t_k)$, $\lambda_k = \lambda (t_k)$, an implicit finite-difference scheme for \eqref{E:prox1} with centered second-order variation gives
\begin{equation}\label{algo3}
\frac{1}{h^2}(x_{k+1} -2 x_{k}   + x_{k-1} ) +\frac{\alpha}{kh^2}( x_{k}  - x_{k-1})  + A_{\lambda_k} (x_{k+1})  =0.
\end{equation}
After expanding \eqref{algo3}, we obtain 
\begin{equation}\label{algo4}
x_{k+1} + h^2 A_{\lambda_k} (x_{k+1}) =  x_{k} + \left( 1- \frac{\alpha}{k}\right) ( x_{k}  - x_{k-1}).
\end{equation}
Setting $s=h^2$, we have
\begin{equation}\label{algo4b}
x_{k+1} = \left(I + s A_{\lambda_k}\right)^{-1}  \left( x_{k} + \left( 1- \frac{\alpha}{k}\right) ( x_{k}  - x_{k-1})\right),
\end{equation}
where $\left(I + s A_{\lambda_k}\right)^{-1} $ is the resolvent of index $s>0$ of the maximally monotone operator $A_{\lambda_k}$.
This gives the following  algorithm 
\begin{equation}\label{algo7}
 \left\{
\begin{array}{rcl}
y_k & =  & x_{k} + \left( 1- \frac{\alpha}{k}\right) ( x_{k}  - x_{k-1}) \\
\rule{0pt}{20pt}
 x_{k+1} & = & \left(I + s A_{\lambda_k}\right)^{-1}  \left( y_k \right). 
 \end{array}\right.
\end{equation}
Using equality (\ref{E:Y4}):
$$
\left( A_{\lambda}\right)_s = A_{\lambda + s},
$$
we can reformulate this last equation as 
$$
\left(I + s A_{\lambda}\right)^{-1}  = \frac{\lambda}{\lambda +s}I + \frac{s}{\lambda +s}\left(I + (\lambda +s) A\right)^{-1}.
$$
Hence, \eqref{algo7} can be rewritten as
$$
\mbox{\rm (RIPA)} \qquad \left\{
\begin{array}{rcl}
y_k&= &  \displaystyle{x_{k} + \left( 1- \frac{\alpha}{k}\right)( x_{k}  - x_{k-1})} \\
\rule{0pt}{20pt}
 x_{k+1} & = & \displaystyle{\frac{\lambda_k}{\lambda_k +s}y_k + \frac{s}{\lambda_k +s}\left(I + (\lambda_k +s) A\right)^{-1} \left( y_k \right)}, 
 \end{array}\right.
$$
where (RIPA) stands for the Regularized  Inertial Proximal Algorithm. 
Let us  reformulate (RIPA) in a compact way. Observe that
\begin{align*}
x_{k+1}  &=  \frac{\lambda_k}{\lambda_k +s}y_k + \frac{s}{\lambda_k +s}\left(I + (\lambda_k +s) A\right)^{-1} \left( y_k \right)\\
 &=  \left(1- \frac{s}{\lambda_k +s}  \right)y_k + \frac{s}{\lambda_k +s}\left(I + (\lambda_k +s) A\right)^{-1} \left( y_k \right)\\
&= y_k -s   \frac{1}{\lambda_k +s} \left( y_k - \left(I + (\lambda_k +s) A\right)^{-1} (y_k) \right).
\end{align*}
By the definition of $A_{\lambda_k + s}$, this is
\begin{equation}\label{algo11}
x_{k+1} = y_k -s A_{\lambda_k + s}\left( y_k \right).
\end{equation}
Thus, setting  $\alpha_k = 1- \frac{\alpha}{k}$, (RIPA) is just
\begin{equation} \label{E:RIPA}
 \left\{
\begin{array}{rcl}
y_k&= &  \displaystyle{x_{k} + \alpha_k( x_{k}  - x_{k-1})} \\
\rule{0pt}{20pt}
 x_{k+1} &=& y_k -s A_{\lambda_k + s}\left( y_k \right),
 \end{array}\right.
\end{equation} 


\begin{remark}
Letting $\lambda_k \to 0$ in (\ref{E:RIPA}), we obtain the classical form of the inertial proximal algorithm 
\begin{equation} \label{E:RIPA_unreg}
\left\{
\begin{array}{rcl}
y_k&= & x_{k} + \alpha_k ( x_{k}  - x_{k-1}) \\
\rule{0pt}{20pt}
 x_{k+1} & = & (I + sA)^{-1} (y_k). 
 \end{array}\right.
\end{equation}
The case $0\leq \alpha_k \leq \bar{\alpha}<1$ has been considered by \'Alvarez-Attouch in \cite{AA1}. The case $\alpha_k \to 1$, which is the most interesting for obtening fast methods (in the line of Nesterov's accelerated methods), and the one that we are concerned with, was recently studied by Attouch-Cabot \cite{AC1,AC2}. In these two papers, the convergence is obtained under the restrictive assumption
$$
\sum_k \alpha_k \|x_{k+1} - x_{k}  \|^2 <+ \infty.
$$
By contrast, our approach, which supposes that $\lambda_k \to +\infty$, provides convergence of the trajectories, without any restrictive assumption on the trajectories. Let us give a geometrical interpretation of (RIPA). As a classical property of the resolvents (\cite[Theorem 23.44]{BC}), for any $x\in \cH$,  $J_{\lambda}x \to \mbox{\rm proj}_S (x)$ as $\lambda \to + \infty$. Thus the algorithm writes
$$
x_{k+1} = \theta_k y_k + (1-\theta_k)
J_{\lambda_k +s  }\left( y_k\right) 
$$
with $\lambda_k \sim + \infty$,   $\theta_k = \frac{\lambda_k}{\lambda_k +s} \sim 1$, and 
$J_{\lambda_k +s  }\left( y_k\right)  \sim \mbox{\rm proj}_S (y_k)$. This is illustrated in the following picture.

\setlength{\unitlength}{6cm}
\begin{picture}(0.5,0.7)(-0.65,0.03)
\put(0.357,0.49){$y_k =  x_{k} + \left(1- \frac{\alpha}{k} \right)( x_{k}  - x_{k-1})$}

\put(0.27,0.4){$x_{k+1} = \theta_k y_k + (1-\theta_k)
J_{\lambda_k +s  }\left( y_k\right) \sim \theta_k y_k +
(1-\theta_k) \mbox{\rm proj}_S (y_k)$}

\put(0.325,0.49){{\tiny $\bullet$}}
\put(0.388,0.58){$x_k$}
\put(0.35,0.57){{\tiny $\bullet$}}
\put(0.398,0.653){$x_{k-1}$}
\put(0.37,0.657){{\tiny $\bullet$}}
\put(0.268,0.451){{\tiny $\bullet $}}
\put(-0.08,0.29){$S$}

\qbezier(-0.15,0.15)(0.4,0.37)(-0.15,0.45)

\put(-0.12,0.215){\line(1,1){0.16}}
\put(-0.02,0.22){\line(1,1){0.136}}
\put(-0.1,0.32){\line(1,1){0.092}}
\put(-0.12,0.39){\line(1,1){0.035}}

\put(0.341,0.504){\vector(-4,-3){0.06}}

\put(0.34,0.505){\line(1,4){0.042}}

\multiput(0.122,0.346)(0.02,0.015){8}
{\line(4,3){0.01}}

\end{picture}

$\mbox{ }$

\end{remark}

\begin{remark} 
As a main difference with (\ref{E:RIPA_unreg}), (RIPA) contains an additional momentum term $\frac{\lambda_k}{\lambda_k +s}y_k $, which enters the definition of $x_{k+1}$.
Although there is some similarity, this is different from the algorithm introduced by Kim and Fessler in \cite{Kim-F}, which also contains an additional momentum term, but which comes within the definition of $y_k$. It is important to mention that, in both cases, the introduction of this momentum term implies additional operations whose cost is negligible.  
\end{remark}

\subsection{Preliminary estimations}

Given $z\in\cH$ and $k\ge 1$, write
\begin{equation}\label{algo12}
h_{z,k} := \frac{1}{2} \|x_k - z\|^2.
\end{equation}
Since there will be no risk of confusion, we shall write $h_k$ for $h_{z,k}$ to simplify the notation. 
The following result is valid for an arbitrary sequence $ \alpha_k \geq 0$:
\begin{lemma}\label{$h_k$ inequality}
Let $\alpha_k \geq 0$. For any $z \in \cH$, the following holds for all $k \geq 1$:
\begin{equation}\label{algo13}
(h_{k+1} - h_k )- \alpha_k (h_{k} - h_{k-1} ) - 
\langle x_{k+1} - y_k , x_{k+1} - z \rangle + \frac{1}{2} \| x_{k+1} - y_k \|^2  = \frac{1}{2} (\alpha_k + {\alpha_k}^2 )\| x_{k} - x_{k-1} \|^2 .
\end{equation}
\end{lemma}
\begin{proof} Since $y_k=   x_{k} + \alpha_k( x_{k}  - x_{k-1})$, we have
\begin{align}
\|  y_k -z\|^2  &=  \|  (x_{k} -z)+ \alpha_k( x_{k}  - x_{k-1}) \|^2   \nonumber \\
 &=  \| x_{k} -z \|^2 + {\alpha_k}^2 \| x_{k}  - x_{k-1} \|^2  +2 \alpha_k \langle x_{k} - z , x_{k}-x_{k-1} \rangle,\nonumber\\
 &=  \| x_{k} -z \|^2 + {\alpha_k}^2 \| x_{k}  - x_{k-1} \|^2  + \alpha_k  \| x_{k} -z \|^2  + \alpha_k\|  x_{k}-x_{k-1} \|^2 - \alpha_k\|  x_{k-1} -z \|^2 \nonumber\\
 &=  \| x_{k} -z \|^2 +  \alpha_k ( \| x_{k} -z \|^2  - \|  x_{k-1} -z \|^2 ) + (\alpha_k + {\alpha_k}^2 )\| x_{k} - x_{k-1} \|^2 \nonumber\\ 
 &= 2(h_k + \alpha_k (h_{k} - h_{k-1} )) +  (\alpha_k + {\alpha_k}^2 )\| x_{k} - x_{k-1} \|^2 .  \label{ineq10}        
\end{align}
Combining (\ref{ineq10}) with the elementary equality
\begin{equation*}
(h_{k+1} -h_k) - \alpha_k (h_{k} - h_{k-1} ) =    
 \frac{1}{2} \| x_{k+1} -z \|^2 -(h_k + \alpha_k (h_{k} - h_{k-1} )    ),    
\end{equation*}
we deduce that
 \begin{align*}
(h_{k+1} -h_k) - \alpha_k (h_{k} - h_{k-1} ) &= 
\frac{1}{2} \| x_{k+1} -z \|^2  -  \frac{1}{2} \| y_{k} -z \|^2   +  \frac{1}{2} (\alpha_k + {\alpha_k}^2 )\| x_{k} - x_{k-1} \|^2 \\
&=  \langle x_{k+1} - y_k , \frac{1}{2} ( x_{k+1} +y_k) -z\rangle + \frac{1}{2} (\alpha_k + {\alpha_k}^2 )\| x_{k} - x_{k-1} \|^2 \\
&=   \langle x_{k+1} - y_k , ( x_{k+1}-z)  + \frac{1}{2}(y_k - x_{k+1})\rangle + \frac{1}{2} (\alpha_k + {\alpha_k}^2 )\| x_{k} - x_{k-1} \|^2 \\
&= \langle x_{k+1} - y_k , x_{k+1} - z \rangle - \frac{1}{2} \| x_{k+1} - y_k \|^2  + \frac{1}{2} (\alpha_k + {\alpha_k}^2 )\| x_{k} - x_{k-1} \|^2
\end{align*}
which gives the claim. 
\end{proof}
Let us now use the specific form of the inertial algorithm (RIPA) and the cocoercivity of the operator $A_{\lambda}$. The following result is the discrete counterpart of Lemma \ref{L:h_z}:
\begin{lemma}\label{$h_k$ inequality-2 }
Let $S = A^{-1}(0) \neq \emptyset$,
and $0 \leq  \alpha_k \leq 1$. For any $z \in S$, the following holds for all $k \geq 1$
\begin{equation}\label{algo14}
(h_{k+1} - h_k )- \alpha_k (h_{k} - h_{k-1} ) +s 
\lambda_k \|  A_{\lambda_k + s}\left( y_k \right) \|^2   \leq \alpha_k \| x_{k} - x_{k-1} \|^2 .
\end{equation}
\end{lemma}
\begin{proof}
By Lemma \ref{$h_k$ inequality} and the fact that $x_{k+1} = y_k -s A_{\lambda_k + s}\left( y_k \right)$
(see (\ref{algo11})), we have
\begin{align*}
&(h_{k+1} - h_k )- \alpha_k (h_{k} - h_{k-1} ) +s 
\langle  A_{\lambda_k + s}\left( y_k \right) , y_k -s A_{\lambda_k + s}\left( y_k \right) - z \rangle + \frac{s^2}{2} \| A_{\lambda_k + s}\left( y_k \right) \|^2 \\  &= \frac{1}{2} (\alpha_k + {\alpha_k}^2 )\| x_{k} - x_{k-1} \|^2 .
\end{align*}
Since $z\in S$, we have $A_{\lambda_k + s}(z)=0  $. By the $(\lambda_k + s)$-cocoercivity property of $A_{\lambda_k + s}$, we deduce that
\begin{equation*} 
\langle  A_{\lambda_k + s}\left( y_k \right) , y_k-z\rangle
 \geq (\lambda_k + s) \| A_{\lambda_k + s}\left( y_k \right) \|^2 .
\end{equation*}
As a consequence,
\begin{align*}
(h_{k+1} - h_k )- \alpha_k (h_{k} - h_{k-1} ) +s 
(\lambda_k + \frac{s}{2} ) \|  A_{\lambda_k + s}\left( y_k \right) \|^2  \leq \frac{1}{2} (\alpha_k + {\alpha_k}^2 )\| x_{k} - x_{k-1} \|^2 .
\end{align*}
But $\frac{1}{2} (\alpha_k + {\alpha_k}^2 ) \leq \alpha_k$ because $0 \leq  \alpha_k \leq 1$. Since $s\ge 0$, the result follows immediately. 
\end{proof}

It would be possible to continue the analysis assuming the right-hand side of (\ref{algo14}) is summable. The main disadvantage of this hypothesis is that it involves the trajectory $(x_k)$, which is unknown. In the following lemma, we show that the two antagonistic terms $s(\lambda_k + \frac{s}{2} ) \|  A_{\lambda_k + s}\left( y_k \right) \|^2 $ and $ \alpha_k \| x_{k} - x_{k-1} \|^2  $
can be balanced, provided $\lambda_k$ is taken large enough.
 
\begin{lemma}\label{$h_k$ inequality-3 }
Let $ S = A^{-1}(0) \neq \emptyset$, and take $\alpha_k =1 -\frac{\alpha}{k}$ with $\alpha >2$. For each $k \geq 1$, set
\begin{equation}\label{algo-basic-assumption}
\lambda_k  = (1+ \epsilon)\frac{s}{\alpha^2}k^2,
\end{equation}
for some $\epsilon >0$, and write $\beta= \frac{1 + \epsilon}{\alpha} $. Then, for each $z \in S$ and all $k \geq 1$, we have
\begin{equation}\label{algo15}
(h_{k+1} - h_k )- \alpha_k (h_{k} - h_{k-1} )  + 
\epsilon  \|x_{k}  - x_{k-1}  \|^2   + \beta k\left( \| x_{k+1} - x_{k}\|^2  -  \| x_{k} - x_{k-1} \|^2  \right) \leq 0.
\end{equation}
\end{lemma}

\begin{proof}
First, rewrite (\ref{algo14}) as 
\begin{align}\label{algo16}
(h_{k+1} - h_k )- \alpha_k (h_{k} - h_{k-1} ) + 
\frac{\lambda_k }{s} \| x_{k+1} - y_k  \|^2   \leq \alpha_k \| x_{k} - x_{k-1} \|^2 .
\end{align}
Let us write $\| x_{k+1} - y_k  \|^2$  in a recursive form. To this end, we use the specific form of $\alpha_k = 1- \frac{\alpha}{k}$ to obtain
\begin{align*}
\| x_{k+1} - y_k  \|^2 &= \| (x_{k+1} - x_{k}) - \alpha_k( x_{k}  - x_{k-1})  \|^2  \\
&= \| (x_{k+1} - x_{k}) - (x_{k} - x_{k-1}   )+ (1-\alpha_k)( x_{k}  - x_{k-1})  \|^2 \\
&= \left\| (x_{k+1} - x_{k}) - (x_{k} - x_{k-1}   )+ \frac{\alpha}{k}( x_{k}  - x_{k-1}) \right\|^2 \\
&=\| (x_{k+1} - x_{k}) - (x_{k} - x_{k-1}   )\|^2 + 
\frac{\alpha^2}{k^2}\|x_{k}  - x_{k-1}  \|^2  + 
2 \frac{\alpha}{k} \langle (x_{k+1} - x_{k}) - (x_{k} - x_{k-1}),  x_{k}  - x_{k-1}  \rangle.
\end{align*}
But
$$
\frac{1}{2}\| x_{k+1} - x_{k}\|^2  = \frac{1}{2} \| x_{k} - x_{k-1} \|^2    + \langle (x_{k+1} - x_{k}) - (x_{k} - x_{k-1}),  x_{k}  - x_{k-1}  \rangle + \frac{1}{2}\| (x_{k+1} - x_{k})-(x_{k} - x_{k-1}) \|^2 .  
$$
By combining the above equalities, we get
\begin{align*}
\| x_{k+1} - y_k  \|^2 
&=\left(1- \frac{\alpha}{k}\right)\| (x_{k+1} - x_{k}) - (x_{k} - x_{k-1}   )\|^2 + 
\frac{\alpha^2}{k^2}\|x_{k}  - x_{k-1}  \|^2  + 
 \frac{\alpha}{k} \left(\| x_{k+1} - x_{k}\|^2  -  \| x_{k} - x_{k-1} \|^2 . \right)
\end{align*}
Using this equality in (\ref{algo16}), and neglecting the nonnegative term $(1- \frac{\alpha}{k})\| (x_{k+1} - x_{k}) - (x_{k} - x_{k-1}   )\|^2$, we obtain
\begin{align}\label{algo17}
(h_{k+1} - h_k )- \alpha_k (h_{k} - h_{k-1} )  + 
 \frac{\lambda_k \alpha^2}{sk^2}  \|x_{k}  - x_{k-1}  \|^2   +  \frac{\lambda_k \alpha}{sk}  \left( \| x_{k+1} - x_{k}\|^2  -  \| x_{k} - x_{k-1} \|^2  \right)
\leq \| x_{k} - x_{k-1} \|^2 .
\end{align}
Using (\ref{algo-basic-assumption}) and the definition $\beta= \frac{1 + \epsilon}{\alpha}$, inequality (\ref{algo17}) becomes (\ref{algo15}).
\end{proof}

\subsection{Main convergence result}

We are now in position to prove the main result of this section, namely:

\begin{Theorem}\label{T:weak_convergence-discrete}
Let $A: \mathcal H \rightarrow 2^{\mathcal H}$ be  a maximally monotone operator such that  $S= A^{-1} (0)\neq \emptyset$. 
Let $(x_k)$ be a sequence generated by the Regularized  Inertial Proximal Algorithm $$
\mbox{\rm (RIPA)} \qquad \left\{
\begin{array}{rcl}
y_k&= &  \displaystyle{x_{k} + \left( 1- \frac{\alpha}{k}\right)( x_{k}  - x_{k-1})} \\
\rule{0pt}{20pt}
 x_{k+1} & = & \displaystyle{\frac{\lambda_k}{\lambda_k +s}y_k + \frac{s}{\lambda_k +s}\left(I + (\lambda_k +s) A\right)^{-1} \left( y_k \right)},
 \end{array}\right.
$$
where $\alpha>2$ and
$$ \lambda_k  = (1+ \epsilon)\frac{s}{\alpha^2}k^2$$
for some $\epsilon > \frac{2}{\alpha -2}$ and all $k \geq 1$. Then,
\begin{itemize}
\item [i)] The speed tends to zero. More precisely, $\| x_{k+1} - x_{k} \| = \mathcal O (\frac{1}{k})$ and $\sum_k  k\| x_{k} - x_{k-1} \|^2 < +\infty$.
\item [ii)] The sequence $(x_k)$ converges weakly, as $k\to+\infty$, to some $\hat x\in S$.
\item [iii)] The sequence $(y_k)$ converges weakly, as $k\to+\infty$, to $\hat x$.
\end{itemize}
\end{Theorem}

\begin{proof}
First, we simplify the writing of the  equation (\ref{algo15}) given in Lemma \ref{$h_k$ inequality-3 }. Setting $g_k:= \|x_{k}  - x_{k-1}  \|^2$, we have 
$$
(h_{k+1} - h_k )- \alpha_k (h_{k} - h_{k-1} )  + 
\epsilon g_k   + \beta k\left( g_{k+1}-  g_k\right) \leq 0.
$$
Then, we multiply by $k$ to obtain
$$
k(h_{k+1} - h_k )- (k-\alpha) (h_{k} - h_{k-1} )  + 
\epsilon kg_k   + \beta k^2\left( g_{k+1}-  g_k\right) \leq 0.
$$
We now write these inequalities in a recursive form, in order to simplify their summation. We have
$$
k(h_{k+1} - h_k )- (k-1) (h_{k} - h_{k-1} )  +(\alpha -1) (h_{k} - h_{k-1} ) 
+ \epsilon kg_k   + \beta k^2 g_{k+1}- \beta (k-1)^2 g_{k} -\beta (2k-1)g_k  \leq 0.
$$
Summing for $p=1,\dots, k$, we obtain
\begin{equation} \label{algo22}
k(h_{k+1} - h_k )  +(\alpha -1) h_{k} 
+ (\epsilon -2 \beta) \sum_1^k  pg_p   + \beta k^2 g_{k+1} +\beta \sum_1^k  g_p  \leq C
\end{equation}
for some positive constant $C$ that depends only on the initial data.
Since $\beta= \frac{1 + \epsilon}{\alpha}$ with $\alpha >2$ and $\epsilon > \frac{2}{\alpha -2}$, we have $\epsilon -2 \beta = \frac{\alpha -2}{\alpha}\left(\epsilon -  \frac{2}{\alpha -2} \right)>0$. From (\ref{algo22}) we infer that
\begin{equation} \label{algo23}
k(h_{k+1} - h_k )  +(\alpha -1) h_{k} 
 \leq C
\end{equation}
for all $k\geq 1$. Since $\alpha >2$ and $h_k\ge 0$, (\ref{algo23}) implies
\begin{equation*} 
k(h_{k+1} - h_k )  + h_{k} \leq C .
\end{equation*}
Equivalently,
$$
kh_{k+1} - (k-1)h_k    \leq C .
$$
Applying this fact recursively, we deduce that
$$
kh_{k+1}  \leq Ck,
$$
which immediately gives $\sup_k \ h_k < +\infty$. Therefore, the sequence $(x_k)$ is bounded. Set
$$
M := \sup_k  \|  x_k \| < +\infty.
$$
Now, (\ref{algo22}) also implies that
\begin{equation} \label{algo27}
k(h_{k+1} - h_k )  + \beta k^2 g_{k+1}  \leq C
\end{equation}
But
\begin{equation} \label{E:h_k/x_k}
|h_{k+1} - h_k | = \frac{1}{2}  \left| \|x_{k+1} - z   \|^2- \|x_{k} - z   \|^2 \right| 
=  \left| \left\langle x_{k+1} - x_{k}, \frac{1}{2} (  x_{k+1}+ x_{k} ) -z\right\rangle \right|
 \leq (M+ \|z\|)\| x_{k+1} - x_{k} \|.
\end{equation} 
Combining this inequality with (\ref{algo27}), and recalling the definition $g_k=\| x_{k} - x_{k-1} \|^2$, we deduce that
$$
\beta \big[k \|x_{k+1}  - x_{k}  \|\big]^2 -(M+ \|z\|)\,\big[k\| x_{k+1} - x_{k} \|\big]-C \leq 0.
$$
This immediately implies that
\begin{equation} \label{algo29}
 \sup_k  k\| x_{k+1} - x_{k} \| < +\infty.
\end{equation}
In other words, $\| x_{k+1} - x_{k} \| = \mathcal O (\frac{1}{k})$. Another consequence of (\ref{algo22}) is that
$$(\epsilon -2 \beta) \sum_{p=1}^k  p\| x_{p} - x_{p-1} \|^2  \leq C - k(h_{k+1} - h_k ).$$
By (\ref{E:h_k/x_k}), we deduce that
$$(\epsilon -2 \beta) \sum_{p=1}^k  p\| x_{p} - x_{p-1} \|^2  \leq C +k(M+ \|z\|)\| x_{k+1} - x_{k} \|.$$
Then, (\ref{algo29}) gives
\begin{equation} \label{algo30}
 \sum_k  k\| x_{k} - x_{k-1} \|^2 < +\infty,
\end{equation}
which completes the proof of item i).

\smallskip

For the convergence of the sequence $(x_k)$, we use Opial's Lemma  \ref{Opial-discrete}. First, since $\alpha_k=1-\frac{\alpha}{k}\le 1$, Lemma \ref{$h_k$ inequality-2 } gives
$$
(h_{k+1} - h_k)+s 
\lambda_k \|  A_{\lambda_k + s}\left( y_k \right) \|^2   \leq \left(1-\frac{\alpha}{k}\right)(h_{k} - h_{k-1}) +\| x_{k} - x_{k-1} \|^2
$$
for all $k \geq 1$. Using (\ref{algo30}) and invoking Lemma \ref{diff-ineq-disc}, we deduce that
\begin{equation}\label{algo32}
\sum_k k\lambda_k \|  A_{\lambda_k + s}\left( y_k \right) \|^2  < + \infty.
\end{equation}
and
$$\sum_k  [h_{k} - h_{k-1}]_+< + \infty.$$
Since $h_k$ is nonnegative, this implies the existence of $\lim_{k\to+\infty}h_k$, and also that of $\lim_{k\to+\infty}\|x_k-z\|$.

In order to conclude using Opial's Lemma \ref{Opial-discrete}, it remains to show that every weak limit point of the sequence $(x_k)$, as $k\to+\infty$, belongs to $S$.
We begin by expressing (\ref{algo32}) with respect to $x_k$, instead of $y_k$. We have
\begin{align*}
\|  A_{\lambda_k +s}\left( x_k \right) \|^2 &\leq
2\|  A_{\lambda_k +s}\left( y_k \right) \|^2
+2 \|  A_{\lambda_k +s}\left( x_k \right) -A_{\lambda_k +s}\left( y_k \right) \|^2 \\
& \leq 2\|  A_{\lambda_k +s}\left( y_k \right) \|^2 + 
\frac{2}{(\lambda_k +s)^2}\| y_k - x_k\|^2 \\
&\leq 2\|  A_{\lambda_k +s}\left( y_k \right) \|^2 + 
\frac{2}{\lambda_k ^2}\|  x_k - x_{k-1}\|^2 ,
\end{align*}
where the last inequality follows from the definition of $y_k$ given in (\ref{algo7}). Using (\ref{algo29}) and the definition of $\lambda_k$, we may find a constant $D\ge 0$ such that
\begin{align*}
\|  A_{\lambda_k +s}\left( x_k \right) \|^2 
& \leq 2\|  A_{\lambda_k +s}\left( y_k \right) \|^2 + 
\frac{D}{k^2\lambda_k^2}. 
\end{align*}
Hence,
$$
\sum_k k\lambda_k \|  A_{\lambda_k + s}\left( x_k \right) \|^2  \leq  2 \sum_k k\lambda_k\|  A_{\lambda_k +s}\left( y_k \right) \|^2 + \sum_k \frac{D}{k\lambda_k}.
$$
Now, using (\ref{algo32}) and the definition of $\lambda_k$, we conclude that
$$
\sum_k k\lambda_k \|  A_{\lambda_k + s}\left( x_k \right) \|^2  < +\infty. 
$$
Since $\lambda_k$ tends to infinity, this immediately implies
$$
\sum_k k(\lambda_k +s)\|  A_{\lambda_k + s}\left( x_k \right) \|^2  < +\infty. 
$$
To simplify the notations, set $\mu_k = \lambda_k +s$, so that
\begin{equation}\label{algo36-d}
\sum_k k\mu_k \|  A_{\mu_k}\left( x_k \right) \|^2  < +\infty. 
\end{equation}
As we shall see, this fact implies
\begin{equation}\label{algo34}
\lim_{k\to +\infty} \| \mu_k A_{{\mu}_k}\left( x_k \right) \|= 0.
\end{equation}
Suppose, for a moment, that this is true. Let $(x_{k_n})_n$ be a subsequence of $(x_k)$ which converges weakly to some $\bar{x}$. We want to prove that $\bar{x} \in S = A^{-1}(0)$. Since $\mu_k$ tends to infinity, we also have
$$\lim_{k\to +\infty} \|  A_{\mu_k}\left( x_k \right) \|= 0.$$
Passing to the limit in
$$
A_{{\mu_{k}}_n}\left( x_{k_n} \right)   \in A (x_{k_n} - \mu_{k_n} A_{\mu_{k_n}}\left( x_{k_n} \right) ),
$$
and using the demi-closedness of the graph of $A$, we obtain 
$$
0 \in A (\bar{x}).
$$
In other words, $\bar{x} \in S$. As a consequence, it only remains to prove (\ref{algo34}) in order to obtain ii) by Opial's Lemma. To this end, define
$$
\omega_k := \| \mu_k A_{\mu_k}\left( x_k \right) \|^2.
$$ 
We intend to prove that $\lim_{k\to+\infty}\omega_k=0$. Using (\ref{algo36-d}) and the definition of $\mu_k$, we deduce that
\begin{equation}\label{algo36}
\sum_k  \frac{1}{k}  \omega_k < + \infty.
\end{equation}
Therefore, if $\lim_{k\to+\infty}\omega_k$ exists, it must be zero. Since
$$\|\mu_k A_{\mu_k}\left( x_k \right) \|=
\|x_k -J_{\mu_kA}(x_k) \| \leq  \|x_k -z\|\le M+\|z\|,$$
we have
\begin{equation}\label{algo40-b}
|\omega_{k+1} -\omega_k | = 
\left| \|\mu_{k+1} A_{\mu_{k+1}}\left( x_{k+1} \right)\|^2 -\|\mu_k  A_{\mu_k}\left( x_k \right) \|^2 \right| 
\leq 2\big(M+\|z\|\big) \|\mu_{k+1} A_{\mu_{k+1}}\left( x_{k+1} \right) -\mu_k  A_{\mu_k}\left( x_k \right) \|.
\end{equation}
On the other hand, by Lemma \ref{L:basic-variation} we have 
\begin{align*}
\|\mu_{k+1} A_{\mu_{k+1}}\left( x_{k+1} \right) -\mu_k  A_{\mu_k}\left( x_k \right) \| &\leq 2 \| x_{k+1} - x_{k} \| + 2 \| x_{k+1}-z \|\frac{|\mu_{k+1} - \mu_{k} |}{\mu_{k+1}} \\
& \leq 2 \| x_{k+1} - x_{k} \| + 2 \| x_{k+1}-z \|\frac{|\lambda_{k+1} - \lambda_{k} |}{\lambda _{k+1}} \\
& \leq 2 \| x_{k+1} - x_{k} \|+ \frac{4(M+\|z\|)}{k+1}, 
\end{align*} 
by the definition of $\lambda_k$. Using (\ref{algo29}) and replacing the resulting inequality in (\ref{algo40-b}), we deduce that there is a constant $E\ge 0$ such that
$$|\omega_{k+1} -\omega_k |\le \frac{E}{k}.$$
But then
$$\sum_k  |(\omega_{k+1})^2 -(\omega_k)^2 |\le E\sum_k  \frac{1}{k}(\omega_{k+1}+\omega_k)<+\infty $$
by (\ref{algo36}). It follows that $\lim_{k\to+\infty}\omega_k^2$ exists and, since $\omega_k\ge 0$, $\lim_{k\to+\infty}\omega_k$ exists as well. This completes the proof of item ii).
Finally, item iii) follows from the fact that $\lim_k\|x_{k+1}-y_k\|=\lim_k\|s A_{\lambda_k + s}\left( y_k \right)\|=0$.
\end{proof}

\subsection{An application to convex-concave saddle value problems}
As shown by R.T. Rockafellar \cite{Rock1},  to each closed convex-convave function $L: X\times Y \to \bar{\mathbb R}$ acting on the product of two Hilbert spaces $X$ and $Y$ is associated a maximally monotone operator $M: X\times Y \rightrightarrows X\times Y $  which is given by $M= (\partial_{x} L, -\partial_{y} L)$.
This makes it possible to convert convex-concave saddle value problems into the search for the zeros of a maximally monotone operator, and thus to apply our results. Let's illustrate it in the case of   the convex constrained structured minimization problem 
$$
{\rm(P)} \quad \min \left\lbrace f(x) + g(y) : Ax-By =0 \right\rbrace ,
$$
where data satisfy the following assumptions:
\begin{itemize}
\item $X,Y,Z$ are real Hilbert spaces
\item  $f : X \to  \mathbb R \cup \{+\infty \}$ and 
$g: Y \to  \mathbb R \cup \left\lbrace +\infty \right\rbrace $  are closed convex proper functions.
\item $A : X \to Z$ and $B : Y \to Z$ are linear continuous operators.
\end{itemize}
Let us first reformulate (P) 
 as a saddle value problem
\begin{equation}\label{eq:saddle1}
\min_{(x,y)\in X \times Y} \max_{z\in Z }
\left\lbrace f(x) + g(y) + \left\langle z, Ax - By \right\rangle \right\rbrace.
\end{equation}
The   Lagrangian  $L : X \times Y \times Z \to  \mathbb R \cup \{+\infty \}$  associated to \eqref{eq:saddle1}
$$
L(x, y, z) = f(x) + g(y) + \left\langle z, Ax - By \right\rangle
$$
is a convex-concave  extended-real-valued function. 
The maximal monotone operator $ M:   X \times Y \times Z \rightrightarrows X \times Y \times Z$ that is associated to $L$ is given by
\begin{equation}\label{descrip1}
M(x, y, z) = \left( \partial_{x,y}L, -\partial_{z}L \right)(x, y, z) = \left(\partial f(x) + A^t z, \ \partial g(y) - B^t z, \ By- Ax \right).
\end{equation}
When the proximal algorithm is applied to the maximally monotone operator $M$, we obtain the so-called proximal method of multipliers. This method was initiated by  Rockafellar \cite{Rock2}. By combining this method with the alternating proximal minimization algorithm for weakly coupled minimization problems, a fully split method is obtained. This approach was successfully developed by Attouch and Soueycatt in \cite{AS}. Introducing inertial terms in this algorithm, as given by (RIPA),  is a subject of further study, which is part of the active research on the acceleration of the (ADMM) algorithms.

\section{Stability with respect to perturbations, errors} \label{Stab}
In this section, we discuss the stability of the  convergence results with respect to external perturbations. We first consider the continuous case, then the corresponding algorithmic results.
\subsection{The continuous case}\label{Stab-cont}
The continuous dynamics is now written in the following form
\begin{equation} \label{E:system-pert}
\ddot{x}(t)   + \frac{\alpha}{t}  \dot{x}(t)     +A_{\lambda(t)}(x(t)) = f(t), 
\end{equation}
where, depending on the context, $f$ can be interpreted as a source term, a perturbation, or an error. We suppose that $f$ is locally integrable to ensure existence and uniqueness for the corresponding Cauchy problem (see Lemma \ref{existence-uniqueness} in the Appendix).
Assuming that $f(t)$ tends to zero fast enough as $t \to + \infty$, we will see that the convergence results proved in Section \ref{S:minimizing} are still valid for the perturbed dynamics (\ref{E:system-pert}). Due to its similarity to the unperturbed case, we give the main lines of the proof, just highlighting the differences. 

\begin{Theorem}\label{T:weak_convergence-pert}
Let $A: \mathcal H \rightarrow 2^{\mathcal H}$ be  a maximally monotone operator such that  $S= A^{-1} (0)\neq \emptyset$. Let 
$x:[t_0,+\infty[\to\mathcal H$ be a solution of the continuous dynamic (\ref{E:system-pert}), where $\alpha>2$ and
$\lambda (t) = (1+\epsilon) \frac{t^2}{\alpha^2}$ with $\epsilon > \frac{4}{\alpha -2}$. Assume also that $\int_{t_0}^{+\infty} t^3 \|f(t)\|^2 dt < + \infty$ and  $\int_{t_0}^{+\infty} t \|f(t)\| dt < + \infty$. 
Then, $x(t)$ converges weakly, as $t\to+\infty$, to an element of $S$. Moreover $\|\dot{x}(t)\| = \mathcal O (1/t)$.
\end{Theorem}

\begin{proof}
First, a similar computation as in Lemma \ref{L:h_z} gives
\begin{equation} \label{E:pert2}
\ddot{h}_z(t) + \frac{\alpha}{t} \dot{h}_z(t) + \lambda (t) \| A_{\lambda(t)}(x(t))  \|^2 \leq \|\dot{x}(t)\|^2 + \|x(t)-z\| \|f(t)\|.
\end{equation} 
Following the arguments in the proof of Lemma \ref{L:h_z-2}, we use (\ref{E:system-pert}), then develop and simplify (\ref{E:pert2}) to obtain 
\begin{equation} \label{E:pert4}
\ddot{h}_z(t) + \frac{\alpha}{t} \dot{h}_z(t) + \epsilon g(t)   + \beta t \dot{g}(t) \leq \|x(t)-z\| \|f(t)\| +
2\beta t\|\dot{x}(t)\| \|f(t)\|,
\end{equation}
where, as in the proof of Proposition \ref{P:system_properties}, we have set $g(t) = \|\dot{x}(t)\|^2$. Using the fact that
$$
2\beta t\|\dot{x}(t)\| \|f(t)\| \leq \frac{\epsilon}{2}\|\dot{x}(t)\|^2  + \frac{2\beta^2}{\epsilon}t^2 \| \|f(t)\|^2 ,
$$
and multiplying by $t$, we obtain
\begin{equation} \label{E:pert5}
t\ddot{h}_z(t) + \alpha \dot{h}_z(t) + \frac{\epsilon}{2} t  g(t)   + \beta t^2 \dot{g}(t)\leq  t\|x(t)-z\| \|f(t))\| +
\frac{2\beta^2}{\epsilon}t^3  \|f(t)\|^2 .
\end{equation}
Integration from $t_0$ to $t$ yields
\begin{equation} \label{E:pert7}
t\dot{h}_z(t) +(\alpha -1)h_z(t) + \beta t^2 g(t) + \left(\frac{\epsilon}{2} - 2 \beta\right)\int_{t_0}^t s g(s) ds \leq C + \int_{t_0}^t s\|x(s)-z\| \|f(s)\| ds + \frac{2\beta^2}{\epsilon} \int_{t_0}^t s^3  \|f(s)\|^2 ds
\end{equation}
for some positive constant $C$ that depends only on the initial data. In all that follows, $C$ is a generic notation for a constant.
By the definition $\beta= \frac{1 + \epsilon}{\alpha}$ and the assumptions on the parameters $\alpha$ and $\epsilon$, we see that $\frac{\epsilon}{2} - 2 \beta >0$. Taking into account also the hypothesis $\int_{t_0}^{+\infty} t^3 \|f(t)\|^2 dt < + \infty$, we deduce that
$$
t\dot{h}_z(t) +(\alpha -1)h_z(t)  \leq C + \int_{t_0}^t s\|x(s)-z\| \|f(s)\| ds .
$$	%
Multiply this expression by $t^{\alpha -2}$, integrate from $t_0$ to $t$, and use Fubini's Theorem to obtain
$$
\frac{1}{2}\|x(t)-z\|^2 \leq C  + \frac{1}{\alpha -1}\int_{t_0}^t   t\|x(t)-z\| \|f(t)\| dt .
$$
The main difference with Section \ref{S:minimizing} is here. We apply Gronwall's Lemma (see \cite[Lemma A.5]{Bre1}) to get
$$
\|x(t)-z\|\leq C  + 
\frac{1}{\alpha -1}\int_{t_0}^t   t \|f(t)\| dt .
$$
Since $\int_{t_0}^{+\infty} t \|f(t)\| dt < + \infty$, we deduce that the trajectory $x(\cdot)$ is bounded. The rest of the proof is essentially the same. First, we obtain 
$$
\sup_{t\geq t_0} t\|\dot{x}(t)\| < + \infty,
$$
by bounding that quantity between the roots of a quadratic expression. Then, we go back to (\ref{E:pert7}) to get that
$$
\int_{t_0}^{+ \infty} t \|\dot{x}(t)\|^2 dt < + \infty.
$$
We use Lemma \ref{basic-edo} to deduce that $\lim_{t\to + \infty} h_z (t)$ exists and
$$
\int_{t_0}^{+ \infty} t \lambda (t) \| A_{\lambda(t)}(x(t))  \|^2  dt < + \infty .
$$
The latter implies $\lim_{t\to+\infty}\lambda(t)\|A_{\lambda(t)}(x(t))\|=0$, and we conclude by means of Opial's Lemma \ref{Opial}.
\end{proof}

\subsection{The algorithmic case}\label{Stab-algo}
Let us first consider how the introduction of the external perturbation $f$ into the continuous dynamics modifies the corresponding algorithm.
Setting $f_k = f(kh)$, a discretization similar to that of the unperturbed case gives
$$\frac{1}{h^2}(x_{k+1} -2 x_{k}   + x_{k-1} ) +\frac{\alpha}{kh^2}( x_{k}  - x_{k-1})  + A_{\lambda_k} (x_{k+1})  = f_k.
$$
After expanding this expression, and setting $s=h^2$, we obtain 
$$
x_{k+1} + s A_{\lambda_k} (x_{k+1}) =  x_{k} + \left( 1- \frac{\alpha}{k}\right) ( x_{k}  - x_{k-1})+sf_k ,
$$
which gives 
$$
\left\{
\begin{array}{rcl}
y_k & =  & x_{k} + \left( 1- \frac{\alpha}{k}\right) ( x_{k}  - x_{k-1}) \\
\rule{0pt}{20pt}
x_{k+1} & = & \left(I + s A_{\lambda_k}\right)^{-1}  \left( y_k +sf_k \right). 
\end{array}\right.
$$
Using the resolvent equation (\ref{E:Y4}), we obtain  the Regularized  Inertial Proximal Algorithm with perturbation
$$
\mbox{\rm (RIPA-pert)} \qquad \left\{
\begin{array}{rcl}
y_k&= &  \displaystyle{x_{k} + \left( 1- \frac{\alpha}{k}\right)( x_{k}  - x_{k-1})} \\
\rule{0pt}{15pt}
x_{k+1} & = & \displaystyle{\frac{\lambda_k}{\lambda_k +s} (y_k + sf_k ) +\frac{s}{\lambda_k +s}J_{(\lambda_k +s)A}  \left( y_k +sf_k\right)}.
\end{array}\right.
$$
Setting  $\alpha_k = 1- \frac{\alpha}{k}$, and with help of the Yosida approximation, this can be written in a compact way as
\begin{equation} \label{E:RIPAb}
\left\{
\begin{array}{rcl}
y_k&= &  \displaystyle{x_{k} + \alpha_k( x_{k}  - x_{k-1})} \\
\rule{0pt}{20pt}
x_{k+1} &=& (y_k +sf_k) -s A_{\lambda_k + s}\left( y_k +sf_k\right).
\end{array}\right.
\end{equation} 
When $f_k=0$ we recover (RIPA). The convergence of $\mbox{\rm (RIPA-pert)}$ algorithm is analyzed in the following theorem. 

\begin{Theorem}\label{T:discrete-pert}
	Let $A: \mathcal H \rightarrow 2^{\mathcal H}$ be  a maximally monotone operator such that  $S= A^{-1} (0)\neq \emptyset$. 
	Let $(x_k)$ be a sequence generated by the algorithm $\mbox{\rm (RIPA-pert)}$
	where $\alpha>2$ and
	$$ \lambda_k  = (1+ \frac{s}{2}+ \epsilon)\frac{2s}{\alpha^2}k^2$$
	for some $\epsilon > \frac{2 +s}{\alpha -2}$ and all $k \geq 1$. Suppose that $\sum_k k\|f_k\| < + \infty$ and   
	$\sum_k k^3\|f_k\|^2 < + \infty$.  Then,
	\begin{itemize}
		\item [i)] The speed tends to zero. More precisely, $\| x_{k+1} - x_{k} \| = \mathcal O (\frac{1}{k})$ and $\sum_k  k\| x_{k} - x_{k-1} \|^2 < +\infty$.
		\item [ii)] The sequence $(x_k)$ converges weakly, as $k\to+\infty$, to some $\hat x\in S$.
		\item [iii)] The sequence $(y_k)$ converges weakly, as $k\to+\infty$, to $\hat x$.
	\end{itemize}
\end{Theorem}
\begin{proof}
	Let us observe that the definitions of  $y_k$ and $h_k$ are  the same as in the unperturbed case. Hence, it is only when using the constitutive
	equation $x_{k+1} = (y_k +sf_k) -s A_{\lambda_k + s}\left( y_k +sf_k\right)$,  which contains the perturbation term, that changes occur in the proof. 
	Thus, the beginning of the proof and Lemma \ref{$h_k$ inequality} is still valid, which gives
	\begin{equation}\label{algo13b}
	(h_{k+1} - h_k )- \alpha_k (h_{k} - h_{k-1} ) - 
	\langle x_{k+1} - y_k , x_{k+1} - z \rangle + \frac{1}{2} \| x_{k+1} - y_k \|^2  = \frac{1}{2} (\alpha_k + {\alpha_k}^2 )\| x_{k} - x_{k-1} \|^2 .
	\end{equation}
	The next step, which corresponds to Lemma \ref{$h_k$ inequality-2 }, uses the constitutive equation. Let us  adapt  it to our situation.
	By (\ref{algo13b}) and  $x_{k+1} - y_k = s(f_k - A_{\lambda_k + s}\left( y_k +sf_k\right))$, it follows that
	\begin{align}\label{algo13c}
	(h_{k+1} - h_k )- &\alpha_k (h_{k} - h_{k-1} ) +s 
	\langle  A_{\lambda_k + s}\left( y_k +sf_k \right) -f_k, (y_k +sf_k -z)-s A_{\lambda_k + s}\left( y_k +sf_k\right) \rangle \nonumber \\
	&+ \frac{s^2}{2} \| A_{\lambda_k + s}\left( y_k +sf_k\right) -f_k\|^2  = \frac{1}{2} (\alpha_k + {\alpha_k}^2 )\| x_{k} - x_{k-1} \|^2 .
	\end{align}
	Since $z\in S$, we have $A_{\lambda_k + s}(z)=0  $. By the $(\lambda_k + s)$-cocoercivity property of $A_{\lambda_k + s}$, we deduce that
	\begin{equation*} 
	\langle  A_{\lambda_k + s}\left( y_k +sf_k\right) , y_k +sf_k-z\rangle
	\geq (\lambda_k + s) \| A_{\lambda_k + s}\left( y_k +sf_k \right) \|^2 .
	\end{equation*}
	Using the above inequality in (\ref{algo13c}), and after development and simplification, we obtain
	\begin{align}\label{basic-algo-pert-0}
	(h_{k+1} - h_k )- \alpha_k (h_{k} - h_{k-1} )& +s 
	(\lambda_k + \frac{s}{2} ) \|  A_{\lambda_k + s}\left( y_k +sf_k \right) \|^2  \nonumber \\
	&\leq \frac{s^2}{2}\| f_k \|^2 + s\langle   y_k -z , f_k\rangle + \frac{1}{2} (\alpha_k + {\alpha_k}^2 )\| x_{k} - x_{k-1} \|^2 .
	\end{align}
	From $ A_{\lambda_k + s}\left( y_k +sf_k \right)= \frac{1}{s}( y_k +sf_k -x_{k+1})$ it follows that
$$(h_{k+1} - h_k )- \alpha_k (h_{k} - h_{k-1} ) +\frac{1}{s} 
(\lambda_k + \frac{s}{2} ) \| (y_k -x_{k+1})+ sf_k \|^2  \leq \frac{s^2}{2}\| f_k \|^2 + s\langle   y_k -z , f_k\rangle + \frac{1}{2} (\alpha_k + {\alpha_k}^2 )\| x_{k} - x_{k-1} \|^2 .$$	
Using the elementary inequality $	\| y_k -x_{k+1} \|^2 \leq 2 \| (y_k -x_{k+1})+ sf_k \|^2  + 2\| sf_k \|^2$,	we deduce that
	\begin{align*}
	(h_{k+1} - h_k )- \alpha_k (h_{k} - h_{k-1} )& +\frac{1}{2s} 
	(\lambda_k + \frac{s}{2} ) \| x_{k+1}-y_k \|^2  \\
	&\leq \frac{s^2}{2}\| f_k \|^2 +s(\lambda_k + \frac{s}{2} )\| f_k \|^2 + s\langle   y_k -z , f_k\rangle + \frac{1}{2} (\alpha_k + {\alpha_k}^2 )\| x_{k} - x_{k-1} \|^2 .
	\end{align*}
	Since $\frac{1}{2} (\alpha_k + {\alpha_k}^2 ) \leq \alpha_k \leq 1$,  $s>0$, and by using Cauchy-Schwarz inequality, it follows that
	\begin{align}\label{algo-pert1}
	(h_{k+1} - h_k )- \alpha_k (h_{k} - h_{k-1} ) &+\frac{\lambda_k}{2s} 
	\| x_{k+1} -y_k  \|^2
	\leq s(s+\lambda_k)\| f_k \|^2 + s\| y_k -z\| \| f_k\|+ \| x_{k} - x_{k-1} \|^2 .
	\end{align}
	By the definition of $y_k$ and elementary inequalities we have
	$$
	\| y_k -z\| \| f_k\| \leq \| x_k -z\| \| f_k\| + \| x_k -x_{k-1}\| \| f_k\| \leq \| x_k -z\| \| f_k\| + \frac{1}{2}\| f_k\|^2 + \frac{1}{2}\| x_k -x_{k-1}\|^2 .
	$$
	Combining this inequality with (\ref{algo-pert1}) we obtain
\begin{equation} \label{algo-pert1b}
(h_{k+1} - h_k )- \alpha_k (h_{k} - h_{k-1} ) +\frac{\lambda_k}{2s} 
\| x_{k+1} -y_k  \|^2 \leq s(\frac{1}{2} +s+\lambda_k)\| f_k \|^2 + s\| x_k -z\| \| f_k\|+ (1+\frac{s}{2}) \| x_{k} - x_{k-1} \|^2 .
\end{equation} 
	Let us write $\| x_{k+1} - y_k  \|^2$  in a recursive form. The same computation as in Lemma \ref{$h_k$ inequality-3 } gives
	\begin{align*}
	\| x_{k+1} - y_k  \|^2 
	&=\left(1- \frac{\alpha}{k}\right)\| (x_{k+1} - x_{k}) - (x_{k} - x_{k-1}   )\|^2 + 
	\frac{\alpha^2}{k^2}\|x_{k}  - x_{k-1}  \|^2  + 
	\frac{\alpha}{k} \left(\| x_{k+1} - x_{k}\|^2  -  \| x_{k} - x_{k-1} \|^2  \right).
	\end{align*}
	Using this equality in (\ref{algo-pert1b}), and neglecting the nonnegative term $(1- \frac{\alpha}{k})\| (x_{k+1} - x_{k}) - (x_{k} - x_{k-1}   )\|^2$, we obtain
	\begin{align}\label{algo17b}
	(h_{k+1} - h_k )- \alpha_k (h_{k} - h_{k-1} )  &+ 
	\frac{\lambda_k \alpha^2}{2sk^2}  \|x_{k}  - x_{k-1}  \|^2   +  \frac{\lambda_k \alpha}{2sk}  \left( \| x_{k+1} - x_{k}\|^2  -  \| x_{k} - x_{k-1} \|^2  \right)
	\nonumber\\
	&\leq  s(\frac{1}{2} +s+\lambda_k)\| f_k \|^2 + s\| x_k -z\| \| f_k\|+ (1+\frac{s}{2}) \| x_{k} - x_{k-1} \|^2  .
	\end{align}
	Using $ \lambda_k  = (1+ \frac{s}{2}+ \epsilon)\frac{2s}{\alpha^2}k^2$ and the definition $\beta= \frac{1 + \frac{s}{2} +\epsilon}{\alpha}$, inequality (\ref{algo17b}) becomes 
	\begin{align}\label{algo18b}
	(h_{k+1} - h_k )- \alpha_k (h_{k} - h_{k-1} )  &+ 
	\epsilon  \|x_{k}  - x_{k-1}  \|^2   +  \beta k  \left( \| x_{k+1} - x_{k}\|^2  -  \| x_{k} - x_{k-1} \|^2  \right)
	\nonumber\\
	&\leq  s(\frac{1}{2} +s+\lambda_k)\| f_k \|^2 + s\| x_k -z\| \| f_k\|.
	\end{align}
	Setting $g_k:= \|x_{k}  - x_{k-1}  \|^2$, we have 
	$$
	(h_{k+1} - h_k )- \alpha_k (h_{k} - h_{k-1} )  + 
	\epsilon g_k   + \beta k\left( g_{k+1}-  g_k\right)
	\leq    s(\frac{1}{2} +s+\lambda_k)\| f_k \|^2 + s\| x_k -z\| \| f_k\|.
	$$
	Then, we multiply by $k$ to obtain
	$$
	k(h_{k+1} - h_k )- (k-\alpha) (h_{k} - h_{k-1} )  + 
	\epsilon kg_k   + \beta k^2\left( g_{k+1}-  g_k\right)  \leq  sk(\frac{1}{2} +s+\lambda_k)\| f_k \|^2 + sk\| x_k -z\| \| f_k\|.
	$$
	We now write these inequalities in a recursive form, in order to simplify their summation. We have
	\begin{align*}
	k(h_{k+1} - h_k )- (k-1) (h_{k} - h_{k-1} )  &+(\alpha -1) (h_{k} - h_{k-1} ) 
	+ \epsilon kg_k   + \beta k^2 g_{k+1}- \beta (k-1)^2 g_{k} -\beta (2k-1)g_k   \\
	& \leq sk(\frac{1}{2} +s+\lambda_k)\| f_k \|^2 + sk\| x_k -z\| \| f_k\|.
	\end{align*}
	Summing for $p=1,\dots, k$, we obtain
\begin{equation} \label{basic-algo-pert}
	k(h_{k+1} - h_k )  +(\alpha -1) h_{k} 
+ (\epsilon -2 \beta) \sum_1^k  pg_p   + \beta k^2 g_{k+1} +\beta \sum_1^k  g_p  \leq s\sum_1^k p\|x_p -z \| \|f_p\| + s\sum_1^k  p(\frac{1}{2} +s+\lambda_p)\| f_p \|^2.
\end{equation} 
Since $\beta= \frac{1 + \frac{s}{2} +\epsilon}{\alpha}$ with $\alpha >2$ and $\epsilon > \frac{2 +s}{\alpha -2}$, we have $\epsilon -2 \beta = \frac{\epsilon (\alpha -2) - (s+2)}{\alpha}>0$. Moreover by the definition of 	$\lambda_k$ and the assumption $\sum_k k^3\|f_k\|^2 < + \infty$, we have  $s\sum_1^k  p(\frac{1}{2} +s+\lambda_p)\| f_p \|^2 \leq C$ for some positive constant $C$. Whence
	\begin{equation} \label{algo23b}
	k(h_{k+1} - h_k )  +(\alpha -1) h_{k} 
	+ \beta k^2 g_{k+1} \\
	\leq C +\sum_1^k p\|x_p -z \| \|f_p\|
	\end{equation}
	for all $k\geq 1$. Since $\alpha >2$ and $h_k\ge 0$, (\ref{algo23b}) implies
	$$
	kh_{k+1} - (k-1)h_k    \leq C +\sum_1^k p\|x_p -z \| \|f_p\| .
	$$
	By summing the above inequalities, and applying Fubini's Theorem, we deduce that
	$$
	kh_{k+1}  \leq Ck +\sum_1^k  p\|x_p -z \| \|f_p\| (k-p) \leq  Ck  +k\sum_1^k p\|x_p -z \| \|f_p\| .
	$$
	Hence
	$$
	\frac{1}{2} \|x_{k+1} - z\|^2 \leq C +\sum_1^k  p\|x_p -z \| \|f_p\| (k-p) \leq  C  +\sum_1^k p\|x_p -z \| \|f_p\| .
	$$
	Applying the discrete form of the Gronwall's Lemma \ref{d-Gronwall}, and $\sum_k k\|f_k\| < + \infty$, we obtain 
	$\sup_k \ h_k < +\infty$. Therefore, the sequence $(x_k)$ is bounded. The remainder of the proof is pretty much as that of Theorem \ref{T:weak_convergence-discrete}.  We first derive
	\begin{equation} \label{algo29b}
	\sup_k  k\| x_{k+1} - x_{k} \| < +\infty.
	\end{equation}
	Then, we combine (\ref{basic-algo-pert}) with (\ref{algo29b}) to obtain
	\begin{equation} \label{algo30b}
	\sum_k  k\| x_{k} - x_{k-1} \|^2 < +\infty.
	\end{equation}
	Since $\alpha_k=1-\frac{\alpha}{k}\le 1$, inequalities (\ref{basic-algo-pert-0}) and (\ref{algo30b}) give
	$$
	(h_{k+1} - h_k)+s  
	(\lambda_k + \frac{s}{2} ) \|  A_{\lambda_k + s}\left( y_k \right) \|^2   \leq \left(1-\frac{\alpha}{k}\right)(h_{k} - h_{k-1}) +\theta_k 
	$$
	for all $k \geq 1$, and $\sum_{k \in \mathbb N} k\theta_k < +\infty$. Invoking Lemma \ref{diff-ineq-disc}, we deduce that $\lim_{k\to+\infty}h_k$ exists and
	\begin{equation}\label{algo32b}
	\sum_k k\lambda_k \|  A_{\lambda_k + s}\left( y_k \right) \|^2  < + \infty.
	\end{equation}
	We conclude using Opial's Lemma \ref{Opial-discrete} as in the unperturbed case.
\end{proof}

\begin{remark}
The perturbation can be interpreted either as a miscomputation of $y_k$ from the two previous iterates, or as an error due to the fact that the resolvent can be computed at a neighboring point $y_k+sf_k$, rather than $y_k$. Anyway, perturbations of order less than $\frac{1}{k^2}$ are admissible and the convergence properties are preserved.
\end{remark}

\section{Quadratic growth and strong convergence} \label{strong}
%
In this section, we examine the case of a maximally monotone operator  $A$ satisfying a quadratic growth property. More precisely, we assume that there is $\nu>0$ such that 
\begin{equation}\label{E:quad} 
\langle x^*,x-z\rangle\ge \nu\,\hbox{dist}(x,S)^2
\end{equation}
whenever $x^*\in Ax$ and $z\in S$. If $A$ is strongly monotone, then (\ref{E:quad}) holds and $S$ is a singleton. Another example is the subdifferential of a convex function $\Phi$ satisfying a quadratic error bound (see \cite{BNPS}). Indeed, 
$$\langle x^*,x-z\rangle\ge \Phi(x)-\min(\Phi)\ge\nu\,\hbox{dist}(x,S)^2$$
if $x^*\in\partial \Phi(x)$ and $z\in S=\hbox{argmin}(\Phi)$. A particular case is when $A=M^*M$, where $M$ is a bounded linear operator with closed range (if $\Phi(x)=\frac{1}{2}\|Mx\|^2$, then $\nabla\Phi(x)=M^*Mx$). We have,
$$\langle Ax,x-z\rangle =\|M(x-z)\|^2\ge\nu\,\hbox{dist}(x,\hbox{Ker}(M))^2$$
(see \cite[Exercise 2.14]{Bre}). 

We obtain the following convergence result:
\begin{Theorem}\label{T:quadratic growth}
Let $A: \mathcal H \rightarrow 2^{\mathcal H}$ be  a maximally monotone operator satisfying (\ref{E:quad}) for some $\nu>0$ and all $x\in \cH$ and $z\in S$. Let $x:[t_0,+\infty[\to\mathcal H$ be a solution of the continuous dynamic
	\begin{equation*}
	\ddot{x}(t)   + \frac{\alpha}{t}  \dot{x}(t)     +A_{\lambda(t)}(x(t)) = 0, 
	\end{equation*}
	where $\alpha>2$ and
	$$\lambda (t) = (1+\epsilon) \frac{t^2}{\alpha^2} 
	\ \mbox{ with } \ \epsilon > \frac{2}{\alpha -2}.$$
	Then, $\lim_{t\to+\infty}\hbox{dist}(x(t),S)=0$. If, moreover, $S=\{\bar z\}$, then $x(t)$ converges strongly to $\bar z$ as $t\to+\infty$.
\end{Theorem}
\begin{proof}
First, fix $t\ge t_0$ and observe that
\begin{eqnarray*}
\langle A_{\lambda(t)}(x(t)),x(t)-z\rangle & = & \langle A_{\lambda(t)}(x(t)),J_{\lambda(t)A}(x(t))-z\rangle
+\langle A_{\lambda(t)}(x(t)),x(t)-J_{\lambda(t)A}(x(t))\rangle \\
& \ge & \nu\,\hbox{dist}(J_{\lambda(t)A}(x(t)),S)^2+\frac{1}{\lambda(t)}\|x(t)-J_{\lambda(t)A}(x(t))\|^2 \\
& \ge & \nu\,\hbox{dist}(J_{\lambda(t)A}(x(t)),S)^2+\frac{1}{\lambda(t)}\left[(1-\zeta)\|x(t)-z\|^2+\left(1-\frac{1}{\zeta}\right)\|z-J_{\lambda(t)A}(x(t))\|^2\right]
\end{eqnarray*}
for all $\zeta>0$ (we shall select a convenient value later on). In turn, the left-hand side satisfies
\begin{eqnarray*}
\langle A_{\lambda(t)}(x(t)),x(t)-z\rangle 
& \le & \|A_{\lambda(t)}(x(t))\|\,\|x(t)-z\| \\
& \le & \lambda(t)\|A_{\lambda(t)}(x(t))\|^2+\|A_{\lambda(t)}(x(t))\|\,\|J_{\lambda(t)A}(x(t))-z\| \\
& \le & \left(\lambda(t)+\frac{\zeta}{2\nu}\right)\|A_{\lambda(t)}(x(t))\|^2+\frac{\nu}{2\zeta}\|J_{\lambda(t)A}(x(t))-z\|^2.
\end{eqnarray*}
Since $\|x(t)-z\|\ge\hbox{dist}(x(t),S)$, by taking $z$ as the projection of $J_{\lambda(t)A}(x(t))$ onto $S$, and combining the last two inequalities, we obtain
$$\frac{(1-\zeta)}{\lambda(t)}\hbox{dist}(x(t),S)^2+\left(\nu-\frac{\nu}{2\zeta}+\frac{\zeta-1}{\zeta\lambda(t)}\right)\hbox{dist}(J_{\lambda(t)A}(x(t)),S)^2 \le \left(\lambda(t)+\frac{\zeta}{2\nu}\right)\|A_{\lambda(t)}(x(t))\|^2.$$
whenever $0<\zeta<1$. Set
$$\zeta=\frac{2+\lambda(t)\nu}{2(1+\lambda(t)\nu)}=\frac{1}{2}+\frac{1}{2(1+\lambda(t)\nu)},$$
so that
$$0<\frac{1}{2}<\zeta\le \frac{1}{2}+\frac{1}{2(1+\lambda(t_0)\nu)}<1,$$
and
$$\nu-\frac{\nu}{2\zeta}+\frac{\zeta-1}{\zeta\lambda(t)}=\frac{2\zeta\lambda(t)\nu-\lambda(t)\nu+2\zeta-2}{2\zeta\lambda(t)}=\frac{2\zeta(1+\lambda(t)\nu)-(2+\lambda(t)\nu)}{2\zeta\lambda(t)}=0.$$
It follows that
$$\frac{(1-\zeta)}{\lambda(t)}\hbox{dist}(x(t),S)^2\le \left(\lambda(t)+\frac{\zeta}{2\nu}\right)\|A_{\lambda(t)}(x(t))\|^2,$$
and so
$$\hbox{dist}(x(t),S)^2\le \left[\frac{2(1+\lambda(t_0)\nu)}{\lambda(t_0)\nu}\right]\,\lambda(t)\left(\lambda(t)+\frac{1}{2\nu}\right)\|A_{\lambda(t)}(x(t))\|^2.$$
The right-hand side goes to zero by (\ref{E:ineq_h_z-16}).
\end{proof}

A similar result holds for (RIPA), namely:
\begin{Theorem}\label{T:quadratic growth_algo}
Let $A: \mathcal H \rightarrow 2^{\mathcal H}$ be  a maximally monotone operator satisfying (\ref{E:quad}) for some $\nu>0$ and all $x\in \cH$ and $z\in S$. Let Let $(x_k)$ be a sequence generated by the algorithm $\mbox{\rm (RIPA)}$, where $\alpha>2$ and $\lambda_k= (1+\epsilon) \frac{sk^2}{\alpha^2}$ with $\epsilon > \frac{2}{\alpha -2}$. Then, $\lim_{k\to+\infty}\hbox{dist}(x_k,S)=0$. If, moreover, $S=\{\bar z\}$, then $x_k$ converges strongly to $\bar z$ as $k\to+\infty$.
\end{Theorem}

\section{Further conclusions from a keynote example} \label{rot}

Let us illustrate our results in the case where $\mathcal H = \mathbb R^2$, and $A$ is the counterclockwise rotation centered at the origin and with the angle $\frac{\pi}{2}$, namely $A(x,y)=(-y,x)$. This is a model situation for a maximally monotone operator that is not cocoercive. The linear operator $A$ is antisymmetric, that is, $\langle A(x,y),(x,y)\rangle=0$ for all $(x,y)\in\cH$.

\subsection{Critical parameters}
Our results are based on an appropriate tuning of the Yosida approximation parameter. Let us analyze the asymptotic behavior of the solution trajectories of the second-order differential equation
\begin{equation} \label{E:pi2}
\ddot{u}(t)   + \frac{\alpha}{t}  \dot{u}(t)     +A_{\lambda(t)}(u(t)) = 0,
\end{equation}
where $u(t)= (x(t), y(t))$.
Since $0$ is the unique zero of $A$, the question is to find the conditions on $\lambda (t)$ which ensure the convergence of $u(t)$ to zero. An elementary computation gives
$$
A_{\lambda}= \frac{1}{1+\lambda^2}
\left(\begin{array}{cc}
\lambda & -1 \\ 
1 & \lambda
\end{array} \right) .
$$ 
For easy computation, it is convenient to set $z(t)=x(t) +i y(t)$, and work with the equivalent  formulation of the problem in the Hilbert space $\mathcal H = \mathbb C$, equipped with the real Hilbert structure $\langle z_1, z_2\rangle=
Re (z_1 \bar{z_2})$. So, the operator and its Yosida approximation  are given respectively  by $Az =iz$ and $A_{\lambda}z = \frac{\lambda +i}{1 + \lambda^2}z$. Then (\ref{E:pi2}) becomes 
$$
\ddot{z}(t)   + \frac{\alpha}{t}  \dot{z}(t)     +\frac{\lambda +i}{1 + \lambda^2}z(t) = 0.
$$
Passing to the phase space $\mathbb C \times \mathbb C$,
and setting $Z(t)= \big(z(t),\dot z(t)\big)^T$, we obtain the first-order equivalent system
$$
\dot{Z}(t) + M(t) Z(t) =0 \quad
\mbox{ where }
M(t)= 
\left(\begin{array}{cc}
0 & -1 \\ 
\frac{\lambda(t) +i}{1 + \lambda(t)^2} & \frac{\alpha}{t}
\end{array} \right) . 
$$
The asymptotic behavior of the trajectories of this system can be  analyzed by examinating the eigenvalues of the matrix $M(t)$, which are given by
$$
\theta (t) = \frac{\alpha}{2t}\left\lbrace 1 \pm
\sqrt{1- \frac{4t^2}{\alpha^2}\frac{\lambda (t) +i}{1 + \lambda(t)^2} }\right\rbrace .
$$
Let us restrict ourselves to the case $\lambda(t)\sim t^p$. If $p>2$, the eigenvalues $\theta_+$ and $\theta_-$ satisfy
$$
\theta_+(t) \sim \frac{1}{t}\qquad\hbox{and}\qquad \theta_-(t)\sim \frac{1}{t^{p-1}}.
$$
Although the solutions of the differential equation $\dot{v}(t) +  \frac{1}{t}v(t)=0$ converge to $0$, those of $\dot{v}(t) +  \frac{1}{t^{p-1}}v(t)=0$ do not. Thus, to obtain the convergence results of our theorem, we are not allowed to let $\lambda(t)$ tend to infinity at a rate greater than $t^2$, which shows that $t^2$ is a critical size for $\lambda(t)$.

\subsection{A comparative illustration}

As an illustration, we depict solutions of some first- and second-order equations involving the rotation operator $A$, obtained using Scilab's {\tt ode} solver. In all cases, the initial condition at $t_0=1$ is $(10,10)$. For second-order equations, we take the initial velocity as $(0,0)$ in order not to force the system in any direction. When relevant, we take $\lambda(t)=(1+\epsilon)t^2/\alpha^2$ with $\alpha=10$ and $\epsilon=1+2(\alpha-2)^{-1}$. For the constant $\lambda$, we set $\lambda=10$. Table \ref{Tab1} shows the distance to the unique equilibrium $(\bar x,\bar y)=(0,0)$ at $t=100$.

\begin{table}[h]
\begin{center} 
\begin{tabular}{|c|c|c|} 
\hline 
Key & Differential Equation	& Distance to $(0,0)$ at $t=100$ \\
\hline 
(E1) & $\dot x(t)+Ax(t)=0$ & 14.141911 \\
\hline
(E2) & $\ddot x(t)+\frac{\alpha}{t}\dot x(t)+Ax(t)=0$ & 3.186e24 \\
\hline
(E3) & $\dot x(t)+A_{\lambda(t)}(x(t))=0$ & 0.0135184 \\
\hline
(E4) & $\dot x(t)+A_\lambda(x(t))=0$ & 0.0007827 \\
\hline
(E5) & $\ddot x(t)+\frac{\alpha}{t}\dot x(t)+A_{\lambda(t)}x(t)=0$ & 0.000323 \\ \hline 
\end{tabular}
\end{center} 
\caption{Distance to the unique equilibrium for a solution of each equation.} \label{Tab1}
\end{table}
Observe that the final position of the solution of (E5) is comparable to that of (E4), which is a first-order equation governed by the {\em strongly monotone} operator $A_\lambda$. Figure \ref{Fig1} shows the solutions to (E1) and (E2), which do not converge to $(0,0)$, while Figure \ref{Fig2} shows the convergent solutions, corresponding to equations (E3), (E4) and (E5), respectively.

\begin{figure}[h]
	\begin{center}
		\begin{tabular}{cc}
			\includegraphics[width=50mm]{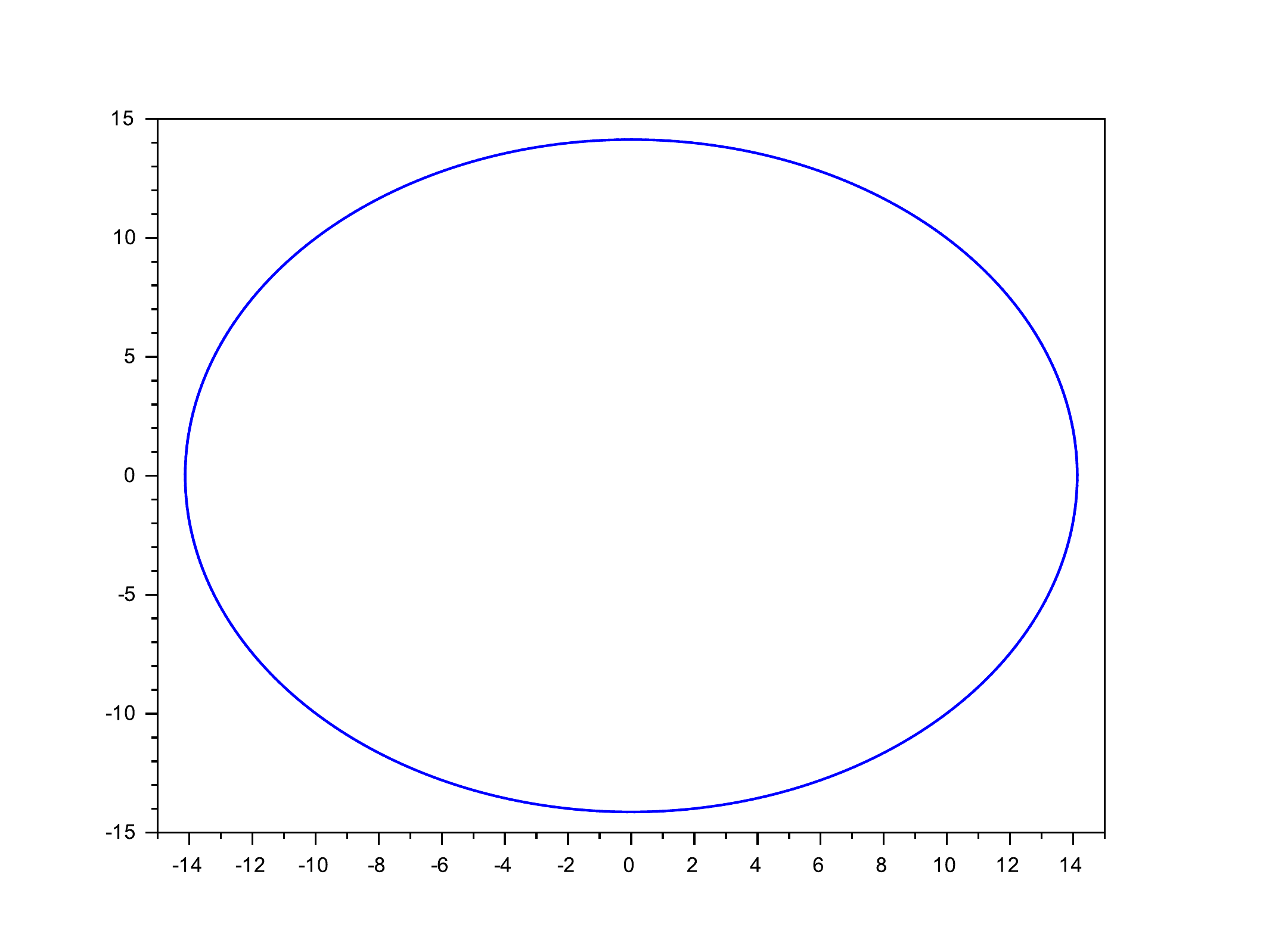} &
			\includegraphics[width=50mm]{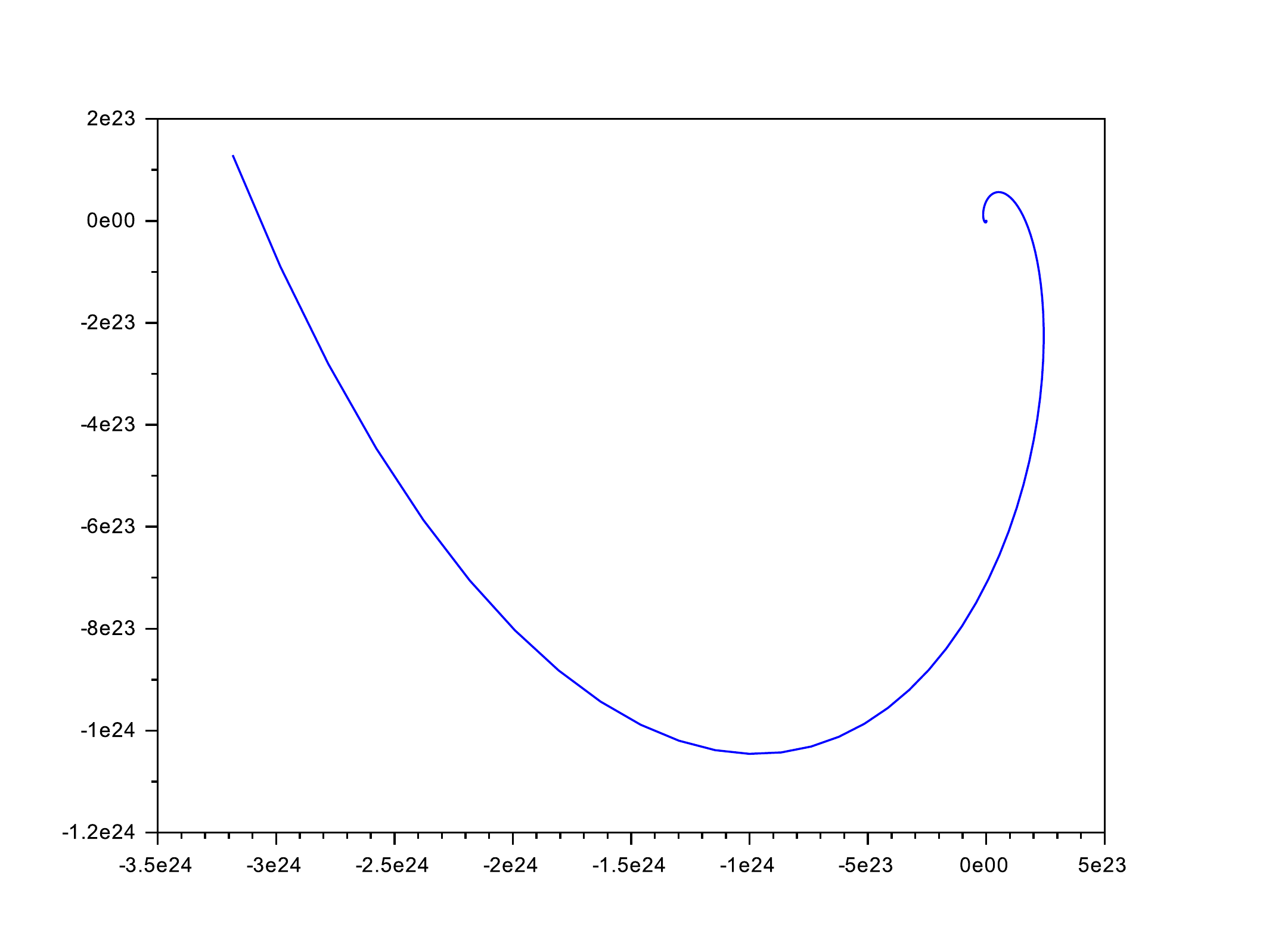}
		\end{tabular}
	\end{center}
	\caption{Solutions of (E1-left) and (E2-right).}
	\label{Fig1}
\end{figure}
\begin{figure}[h]
	\begin{center}
		\begin{tabular}{ccc}
			\includegraphics[width=50mm]{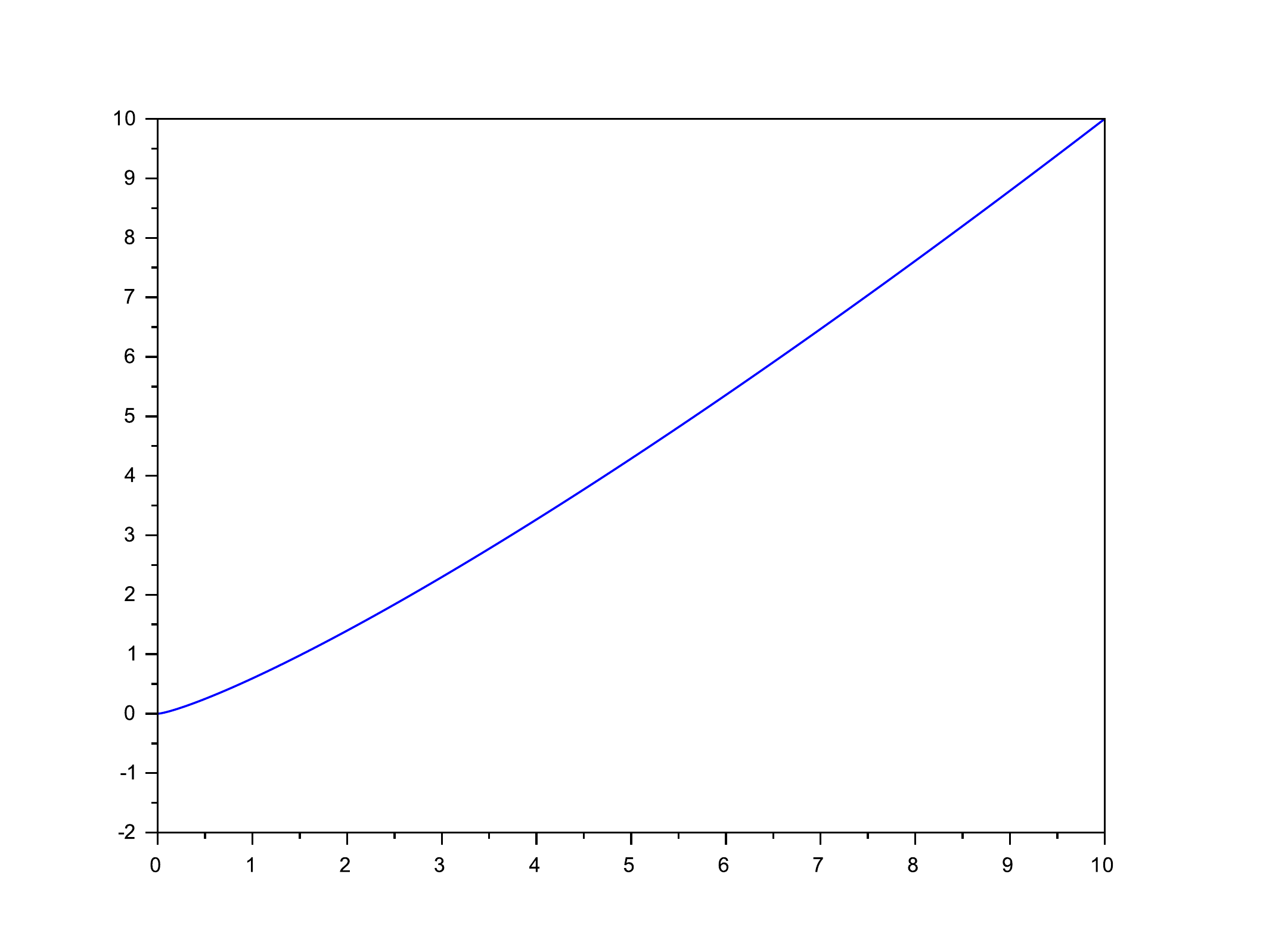} &
			\includegraphics[width=50mm]{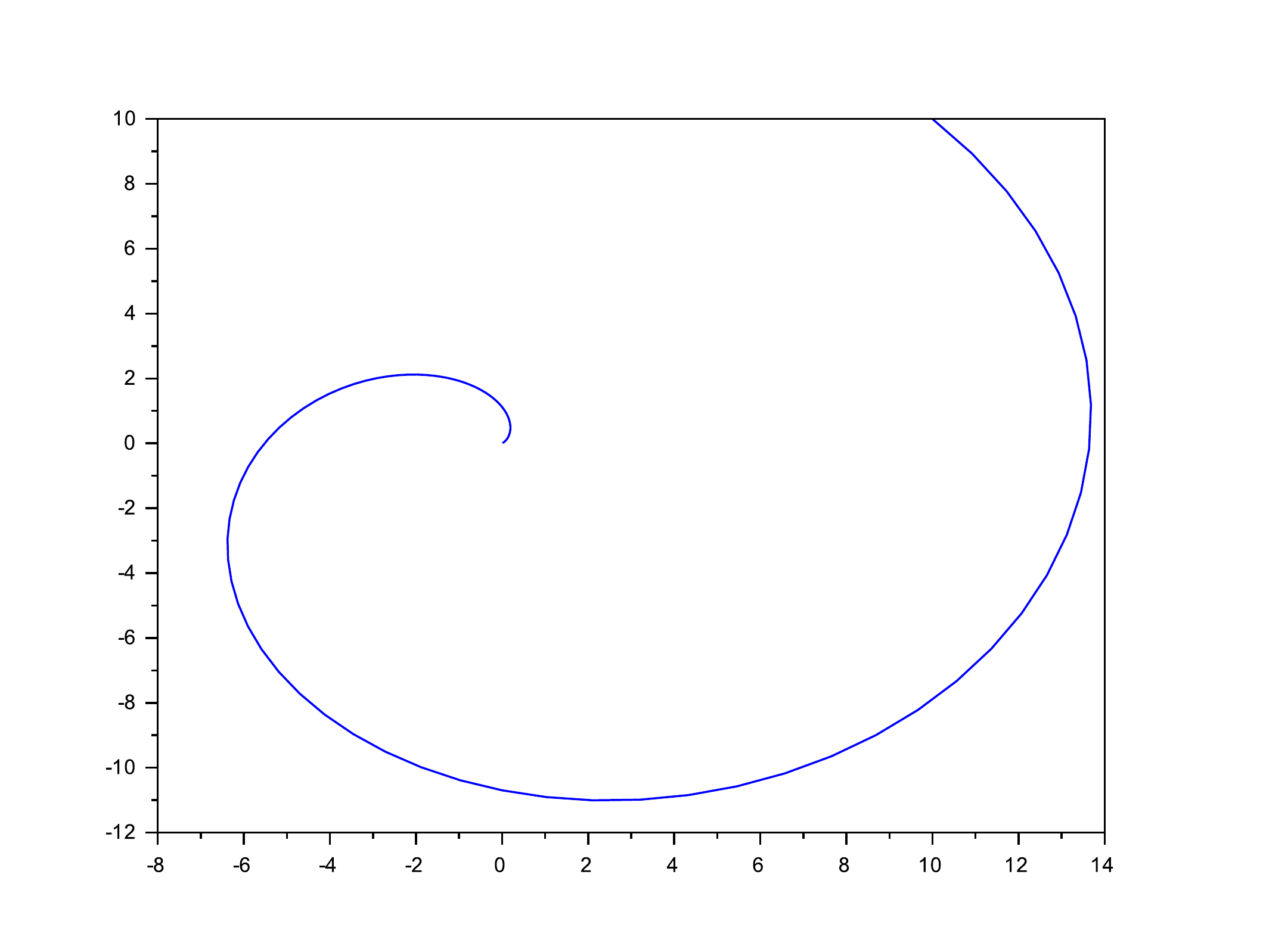} &
			\includegraphics[width=50mm]{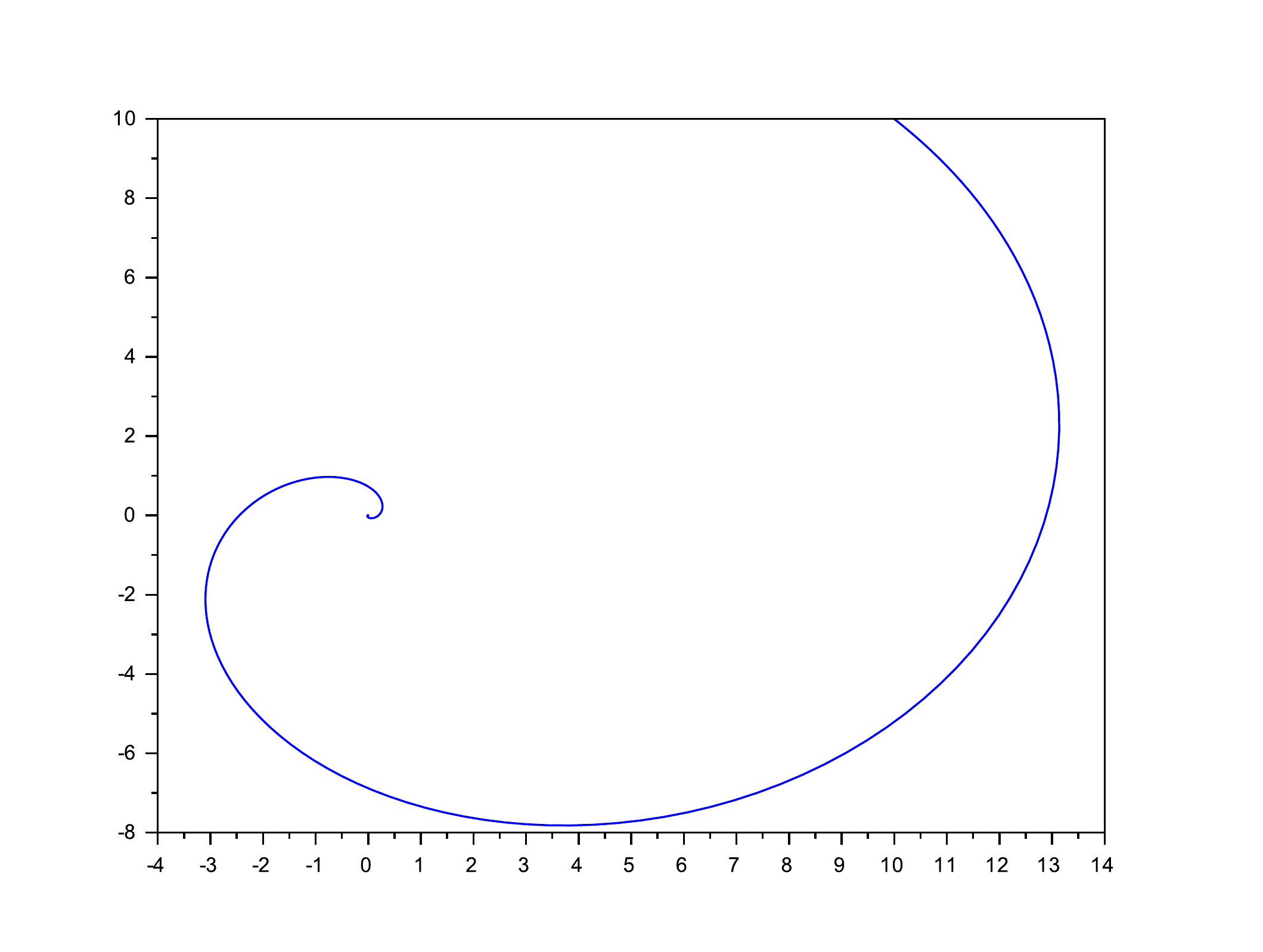}
		\end{tabular}
	\end{center}
	\caption{Solutions of (E3-left), (E4-middle) and (E5-right).}
	\label{Fig2}
\end{figure}

\noindent{\bf Acknowledgement:} The authors thank P. Redont for his careful and constructive reading of the paper.

\appendix
\section{Auxiliary results}

\subsection{Yosida regularization of an operator $A$} \label{SS:Yosida}
Given a maximally monotone operator $A$ and $\lambda >0$, the {\em resolvent} of $A$ with index $\lambda$ and the {\em Yosida regularization} of $A$ with parameter $\lambda$ are defined by
$$
J_{\lambda A} = \left( I + \lambda A \right)^{-1}\qquad\hbox{and}\qquad A_{\lambda} = \frac{1}{\lambda} \left( I- J_{\lambda A} \right),
$$
respectively. The operator  $J_{\lambda A}: \cH \to \cH$ is nonexpansive and eveywhere defined (indeed it is firmly non-expansive). Moreover, $A_{\lambda}$ is $\lambda$-cocoercive: for all $x, y \in \cH$ we have
$$
\langle  A_{\lambda}y -  A_{\lambda}x, y-x\rangle
\geq \lambda \| A_{\lambda}y -  A_{\lambda}x \|^2 .
$$
This property immediately implies that $A_{\lambda}: \cH \to \cH$
is $\frac{1}{\lambda}$-Lipschitz continuous. Another property that proves useful is the {\em resolvent equation} (see, for example, \cite[Proposition 2.6]{Bre1} or \cite[Proposition 23.6]{BC})
\begin{equation} \label{E:Y4}
(A_\lambda)_{\mu}= A_{(\lambda +\mu)},
\end{equation}
which is valid for any $\lambda, \mu >0$. This property allows to compute simply the resolvent of $A_\lambda$: for any $\lambda, \mu >0$ by
$$
J_{\mu A_\lambda} = \frac{\lambda}{\lambda + \mu}I + \frac{\mu}{\lambda + \mu}J_{(\lambda + \mu)A}.
$$
Also note that for any $x \in \cH$, and any $\lambda >0$
$$
A_\lambda (x) \in A (J_{\lambda A}x)= A( x - \lambda  A_\lambda (x)).
$$
Finally, for any $\lambda >0$, $A$ and $A_{\lambda}$ have the same solution set $S:=A_{\lambda}^{-1} (0) = A^{-1}(0)$. For a detailed presentation of the properties of the maximally monotone operators and the Yosida approximation, the reader can consult  \cite{BC} or \cite{Bre1}.

\subsection{Existence and uniqueness of solution in the presence of a source term}

Let us first establish the existence and uniqueness of the solution trajectory of the Cauchy problem associated to the continuous regularized dynamic (\ref{E:system}) with a source term.
\begin{lemma}\label{existence-uniqueness} Take $t_0>0$. Let us suppose that $\lambda : [t_0, +\infty[ \to \mathbb R^+$ is a measurable function such that $\lambda (t) \geq \underline{\lambda}$ for some $\underline{\lambda}>0$. Suppose that $f \in L^1 ([t_0, T], \cH)$ for all $T \geq t_0$. Then,  for any  $x_0 \in \mathcal H,  \ v_0 \in \mathcal H $, there exists  a unique strong global solution $x: [t_0,+\infty[ \to \cH$ of the Cauchy problem 
\begin{equation}\label{CL1}
		\left\{
		\begin{array}{rcl}
		&&\ddot{x}(t)   + \frac{\alpha}{t}  \dot{x}(t)     +A_{\lambda(t)}(x(t)) = f(t)\\
		\rule{0pt}{20pt}
		&& x(t_0)= x_0, \ \dot{x}(t_0)=v_0.
		\end{array}\right.
\end{equation}
\end{lemma}

\begin{proof}
The argument is standard, and consists in writing (\ref{CL1}) as a first-order system in the phase space. By setting 
		$$X(t)= 
		\begin{pmatrix}
		x(t) \\[3mm]
		\dot{x}(t)
		\end{pmatrix}, 
		\quad
		F (t,u,v)= 
		\begin{pmatrix}
		v \\[3mm]
		-\frac{\alpha}{t} v - A_{\lambda(t)}u + f(t)
		\end{pmatrix}
		\mbox{ and }
		X_0= 
		\begin{pmatrix}
		x_0 \\[3mm]
		v_0
		\end{pmatrix},
		$$ 
		the system can be written as
\begin{equation}\label{CL2}
		\left\{
		\begin{array}{rcl}
		&& \dot{X}(t) = F (t,X(t))\\
		\rule{0pt}{20pt}
		&& X(t_0)= X_0 .
		\end{array}\right.
		\end{equation}
		Using the $\frac{1}{\lambda}$-Lipschitz continuity property of $A_{\lambda}$, one can easily verify that the conditions of the Cauchy-Lipschitz theorem are satisfied. Precisely, we can apply the non-autonomous version of this theorem given in 
		\cite[Proposition 6.2.1]{Haraux}.
		Thus, we obtain a strong solution, that is, $t\mapsto\dot{x}(t)$ is locally absolutely continuous.
		If, moreover, we suppose that the functions $\lambda (\cdot)$ and $f$ are continuous, then the solution is a classical solution of class $\mathcal C^2$.
\end{proof}

\subsection{Opial's Lemma} The following results are often referred to as Opial's Lemma \cite{Op}. To our knowledge, it was first written in this form in Baillon's thesis. See \cite{Pey_Sor} for a proof.

\begin{lemma}\label{Opial} Let $S$ be a nonempty subset of $\mathcal H$ and let $x: [0,+\infty [ \to \mathcal H$. Assume that 
\begin{itemize}
\item [(i)] for every $z\in S$, $\lim_{t\to\infty}\|x(t)-z\|$ exists;
\item [(ii)] every weak sequential limit point of $x(t)$, as $t\to\infty$, belongs to $S$.
\end{itemize}
Then $x(t)$ converges weakly as $t\to\infty$ to a point in $S$.
\end{lemma}

Its discrete version is

\begin{lemma}\label{Opial-discrete} Let $S$ be a non empty subset of $\mathcal H$,
and $(x_k)$ a sequence of elements of $\mathcal H$. Assume that 
\begin{itemize}
\item [(i)] for every $z\in S$, $\lim_{k\to+\infty}\|x_k-z\|$ exists;
\item [(ii)] every weak sequential limit point of $(x_k)$, as $k\to\infty$, belongs to $S$.
\end{itemize}
Then $x_k$ converges weakly as $k\to\infty$ to a point in $S$. 
\end{lemma}

\subsection{Variation of the function $\gamma\mapsto \gamma A_{\gamma}x$}
	
	\begin{lemma} \label{L:basic-variation}
		Let $\gamma, \delta >0$, and $x, y\in \cH$. Then,  for each $z\in S= A^{-1} (0)$, and all $t \geq t_0$, we have
		\begin{equation} \label{E:basic-var1}
		\|\gamma A_{\gamma}x - \delta A_{\delta}y\| \leq 2 \| x-y \| + 2 \| x-z \|\frac{|\gamma - \delta |}{\gamma}
		\end{equation} 
	\end{lemma}
\begin{proof} We use successively the definition of the Yosida approximation, the  resolvent identity \cite[Proposition 23.28 (i)]{BC}, and the nonexpansive property of the resolvent, to obtain
\begin{align*}
		\|\gamma A_{\gamma}x - \delta A_{\delta}y\|&\leq \|x-y\| + \| J_{\gamma A}x  - J_{\delta A}y \|\\
		&=  \|x-y\| + \| J_{\delta A}\left( \frac{\delta}{\gamma}x  + \left(1- \frac{\delta}{\gamma}     \right)J_{\gamma A}x\right)     - J_{\delta A}y \|  \\
		& \leq \|x-y\| + \|  \frac{\delta}{\gamma}x  + \left(1- \frac{\delta}{\gamma}     \right)J_{\gamma A}x  -y\|\\
		& \leq 2\|x-y\| + |1- \frac{\delta}{\gamma}  | \| J_{\gamma A}x -x \|.
\end{align*}
Since $J_{\gamma A}z =z$ for $z\in S$, and using again the nonexpansive property of the resolvent, we deduce that 
\begin{align*}
		\|\gamma A_{\gamma}x - \delta A_{\delta}y\|
		& \leq 2\|x-y\| + |1- \frac{\delta}{\gamma}  | \| (J_{\gamma A}x -J_{\gamma A}z) + (z -x \|)\\
		& \leq 2\|x-y\| + 2 \| x-z \|\frac{|\gamma - \delta |}{\gamma},
\end{align*}
which gives the claim.
\end{proof}

\subsection{On integration and decay}
	
\begin{lemma} \label{L:basic-integration}
		Let $w,\eta:[t_0,+\infty[\to[0,+\infty[$ be absolutely continuous functions such that $\eta\notin L^1 (t_0, +\infty)$,
		$$\int_{t_0}^{+ \infty}  w(t)\,\eta(t)\,dt < + \infty,$$
		and $|\dot w(t)| \leq \eta(t)$  for almost every $t>t_0$. Then, 
		$\lim_{t\to+\infty} w(t) =0$.
\end{lemma}
	
\begin{proof}
		First, for almost every $t>t_0$, we have
		$$
		\left|\frac{d}{dt} w^2(t)\right|= 2\left|\frac{d}{dt} w(t)\right| w(t) \leq 
		2w(t)\,\eta(t).
		$$
		Therefore, $|\frac{d}{dt} w^2|$ belongs to $L^1$. This implies that $\lim_{t\to+\infty} w^2(t) $ exists. Since $w$ is nonnegative, it follows that $\lim_{t\to+\infty} w(t) $ exists as well. But this limit is necessarily zero because $\eta\notin L^1$.
\end{proof}

\subsection{On boundedness and anchoring}

\begin{lemma} \label{basic-edo} 
Let $t_0>0$, and let $w: [t_0, +\infty[ \rightarrow \mathbb R$ be a continuously differentiable function which is bounded from below. Given a nonegative function $\theta$, let us assume that
\begin{equation}\label{basic-edo1}
t\ddot{w}(t) + \alpha \dot w(t) + \theta (t)\leq k(t),
\end{equation}
for some $\alpha > 1$, almost every $t>t_0$, and some nonnegative function $k\in L^1 (t_0, +\infty)$. Then, the positive part $[\dot w]_+$ of $\dot w$ belongs to $L^1(t_0,+\infty)$, and $\lim_{t\to+\infty}w(t)$ exists.
Moreover, we have  $\int_{t_0}^{+\infty} \theta (t) dt < + \infty$.
\end{lemma}

\begin{proof} 
Multiply \eqref{basic-edo1} by $t^{\alpha-1}$ to obtain
$$
\frac{d}{dt} \big(t^{\alpha} \dot w(t)\big)+ t^{\alpha-1} \theta (t)\leq t^{\alpha-1} k (t).
$$
By integration, we obtain
\begin{equation}\label{basic-ineq1}
\dot w(t) +  \frac{1}{t^{\alpha} }  \int_{t_0}^t  s^{\alpha-1}\theta (s)ds \leq \frac{{t_0}^{\alpha}|\dot w(t_0)|}{t^{\alpha} }  +  \frac{1}{t^{\alpha} }  \int_{t_0}^t  s^{\alpha-1}k (s)ds.
\end{equation}
Hence, 
$$
[\dot w]_{+}(t)  \leq \frac{{t_0}^{\alpha}|\dot w(t_0)|}{t^{\alpha} }  +  \frac{1}{t^{\alpha} }  \int_{t_0}^t  s^{\alpha-1} k(s)ds,
$$
and so,
$$
\int_{t_0}^{\infty} [\dot w]_{+}(t)  dt \leq \frac{{t_0}^{\alpha}|\dot w(t_0)|}{(\alpha -1) t_0^{\alpha -1}} +
  \int_{t_0}^{\infty}\frac{1}{t^{\alpha}}  \left(  \int_{t_0}^t  s^{\alpha-1} k(s) ds\right)  dt.
$$
Applying Fubini's Theorem, we deduce that
$$
\int_{t_0}^{\infty}\frac{1}{t^{\alpha}}  \left(  \int_{t_0}^t  s^{\alpha-1} k(s) ds\right)  dt =     
\int_{t_0}^{\infty}  \left(  \int_{s}^{\infty}  \frac{1}{t^{\alpha}} dt\right) s^{\alpha-1} k(s) ds 
= \frac{1}{\alpha -1} \int_{t_0}^{\infty}k(s) ds.
$$
As a consequence,
$$
 \int_{t_0}^{\infty} [\dot w]_{+}(t) dt \leq \frac{{t_0}^{\alpha}|\dot w(t_0)  |}{(\alpha -1) t_0^{\alpha -1}}  +
 \frac{1}{\alpha -1} \int_{t_0}^{\infty}k(s)  ds < + \infty.
$$
This implies $\lim_{t\to+\infty}w(t)$ exists. Back to (\ref{basic-ineq1}), integrating from $t_0$ to $t$, using Fubini's Theorem again, and then letting $t$ tend to $+\infty$, we obtain
$$
\lim_{t\to +\infty} w(t) - w(t_0) + \frac{1}{\alpha -1} \int_{t_0}^{\infty}\theta (s)  ds \leq \frac{{t_0}^{\alpha}|\dot w(t_0)  |}{(\alpha -1) t_0^{\alpha -1}}  + \frac{1}{\alpha -1} \int_{t_0}^{\infty}k(s)  ds < + \infty.
$$
Hence $\int_{t_0}^{\infty}\theta (s)  ds < + \infty$.
\end{proof}

\subsection{A summability result for real sequences}

\begin{lemma}\label{diff-ineq-disc} 
Let $\alpha >1$, and let $(h_k)$ be a sequence of real numbers which is bounded from below, and such that
\begin{equation} \label{E:summability}
(h_{k+1} - h_{k}) -  \left(1- \frac{\alpha}{k}\right)    (h_{k} - h_{k-1}) + \omega_k \leq  \theta_k
\end{equation} 
for all $k\geq 1$. 
Suppose that $(\omega_k)$, and  $(\theta_k)$ are two sequences of nonnegative numbers, such  that $\sum_k k\theta_{k} <+\infty$.
Then 
$$\sum_{k \in \mathbb N} [h_k - h_{k-1}  ]_{+} < +\infty \quad \mbox{ and } \quad  \sum_{k \in \mathbb N} k\omega_k < +\infty.
$$
\end{lemma}
 
\begin{proof} Since $(\omega_k)$ is nonegative, we have
$$
(h_{k+1} - h_{k}) -  \left(1- \frac{\alpha}{k}\right)    (h_{k} - h_{k-1})  \leq  \theta_k.
$$
Setting $b_k := [h_k - h_{k-1}  ]_{+}$ the positive part of $h_k - h_{k-1}$, we immediately infer that
$$
b_{k+1} 
\leq  \left(1- \frac{\alpha}{k}\right)b_k + 
\theta_k
$$
for all $k\geq 1$. Multiplying by $k$ and rearranging the terms, we obtain
$$(\alpha-1)b_k\le(k-1)b_k-kb_{k+1}+k\theta_k.$$
Summing for $k=1,\dots, K$, and using the telescopic property, along with the fact that $Kb_{K+1}\ge 0$, we deduce that
$$(\alpha-1)\sum_{k=1}^Kb_k\le \sum_{k=1}^Kk\theta_k,$$
which gives
$$\sum_{k \in \mathbb N} [h_k - h_{k-1}  ]_{+} < +\infty.$$
Let us now prove that $\sum_{k \in \mathbb N} k\omega_k < +\infty$, which is the most delicate part of the proof. To this end, write $\delta_k= h_{k} - h_{k-1}$, and $\alpha_k =\left(1- \frac{\alpha}{k}\right)$, so that \eqref{E:summability} becomes
$$\delta_{k+1} + \omega_k \leq  \alpha_k\delta_k+\theta_k.$$
An immediate recurrence (it can be easily seen by induction) shows that
\begin{equation*}
\delta_{k+1}  +\sum_{i=1}^k\left[\left(\prod_{j=i+1}^k \alpha_j\right)\omega_i\right]\leq \left(\prod_{j=1}^k \alpha_j\right) \delta_1+\sum_{i=1}^k\left[\left(\prod_{j=i+1}^k \alpha_j\right)\theta_i\right],
\end{equation*}
with the convention $\prod_{j=k+1}^k \alpha_j=1$. To simplify the notation, write $A_{i}^k=\prod_{j=i}^k \alpha_j$. Sum the above inequality for $k=1,\dots, K$ to  deduce that
\begin{equation}\label{Lemma A7-1}
h_{K+1}-h_1 +\sum_{k=1}^{K}\sum_{i=1}^kA_{i+1}^k\omega_i \leq \delta_1\sum_{k=1}^{K}A_1^k +\sum_{k=1}^{K}\sum_{i=1}^kA_{i+1}^k\theta_i.
\end{equation}
Now, using Fubini's Theorem, we obtain
\begin{equation} \label{E:summability_Fubini}
h_{K+1}-h_1 +\sum_{i=1}^{K}\left[\omega_i\sum_{k=i}^KA_{i+1}^k\right] \leq \delta_1\sum_{k=1}^{K}A_{1}^k +\sum_{i=1}^{K}\left[\theta_i\sum_{k=i}^KA_{i+1}^k\right].
\end{equation} 
Simple computations (using integrals in the estimations) show that
$$\left(\frac{i}{k}\right)^{\alpha}\le A_{i+1}^k\le\left(\frac{i+1}{k+1}\right)^{\alpha},$$
and
$$\frac{i}{\alpha-1}\le \sum_{k=i}^\infty A_{i+1}^k\le \frac{i}{\alpha-1}\left(\frac{i+1}{i}\right)^{\alpha}$$
(see also \cite{AC2} for further details). Letting $K\to+\infty$ in \eqref{E:summability_Fubini}, we deduce that
$$\sum_{i=1}^{\infty}i\omega_i\le C+D\sum_{i=1}^\infty i\theta_i<+\infty$$
for appropriate constants $C$ and $D$.
\end{proof}

\subsection{A discrete Gronwall lemma}

\begin{lemma} \label{d-Gronwall} 
	Let $c\ge 0$ and let $(a_k)$ and $(\beta_j )$ be nonnegative sequences such that $(\beta_j )$ is summable and
	$$
	a_k^2 \leq c^2 + \sum_{j=1}^k \beta_j  a_j
	$$
	for all $k\in\mathbb N$. Then, $\displaystyle a_k \leq c + \sum_{j=1}^{\infty} \beta_j$ for all $k\in\mathbb N$.
\end{lemma}

\begin{proof}
	For $k\in\mathbb N$, set $A_k := \max_{1\leq m \leq  k} a_m $. Then, for $1\leq m\leq k$, we have
	$$
	a_m^2 \leq c^2 + \sum_{j=1}^m \beta_j  a_j \leq c^2 + A_k \sum_{j=1}^{\infty} \beta_j.  
	$$
	Taking the maximum over $1\leq m\leq k$, we obtain
	$$
	A_k^2 \leq c^2 + A_k \sum_{j=1}^{\infty} \beta_j.
	$$
	Bounding by the roots of the corresponding quadratic equation, we obtain the result.
\end{proof}


\end{document}